\documentclass[]{interact}

\usepackage{epstopdf}
\usepackage[caption=false]{subfig}

\usepackage[numbers,sort&compress]{natbib}
\bibpunct[, ]{[}{]}{,}{n}{,}{,}
\renewcommand\bibfont{\fontsize{10}{12}\selectfont}

\usepackage{hyperref}
\usepackage{amsmath}
\usepackage{amssymb}
\usepackage{mathtools}
\usepackage{amsthm}

\usepackage{booktabs}
\usepackage{diagbox}
\usepackage{algorithm}
\usepackage{algorithmic}

\theoremstyle{plain}
\newtheorem{theorem}{theorem}[section]
\newtheorem{lemma}[theorem]{Lemma}
\newtheorem{corollary}[theorem]{Corollary}

\theoremstyle{definition}
\newtheorem{definition}[theorem]{Definition}

\theoremstyle{remark}
\newtheorem{remark}{Remark}

\newtheorem{assumption}[theorem]{Assumption}

\newtheorem{proofpart}{Part}
\makeatletter
\@addtoreset{proofpart}{theorem}
\makeatother

\renewcommand{\epsilon}{\varepsilon}
\newcommand{\R}{\mathcal{R}}
\newcommand{\Rb}{\mathbb{R}}
\newcommand{\nR}{\nabla \mathcal{R}}

\begin{document}
	
	
	\title{Convergence of the Iterates for Momentum and RMSProp for Local Smooth Functions: Adaptation is the Key}
	
	\author{
		\name{Bilel Bensaid$^\ast$\textsuperscript{a} \thanks{$^\ast$ Email: bilel.bensaid@u-bordeaux.fr}, Ga\"el Po\"ette$\dagger$\thanks{$\dagger$Email: gael.poette@cea.fr} \textsuperscript{a} and Rodolphe Turpault$\ddagger$\textsuperscript{b}\thanks{$\ddagger$Email: rodolphe.turpault@u-bordeaux.fr}}
		\affil{\textsuperscript{a}CEA-CESTA/DAM, F-33114 Le Barp, France \\ \textsuperscript{b}Universit\'e de Bordeaux (CNRS-Bordeaux-INP), F33405 Talence, France}
	}
	
	\maketitle
	
	\begin{abstract}
		Both accelerated and adaptive gradient methods are among state of the art algorithms to train neural networks. The tuning of hyperparameters is needed to make them work efficiently. For classical gradient descent, a general and efficient way to adapt hyperparameters is the Armijo backtracking. The goal of this work is to generalize the Armijo linesearch to Momentum and RMSProp, two popular optimizers of this family, by means of stability theory of dynamical systems. We establish convergence results, under the Lojasiewicz assumption, for these strategies. As a direct result, we obtain the first guarantee on the convergence of the iterates for RMSProp, in the non-convex setting without the classical bounded assumptions.
	\end{abstract}
	

	\begin{amscode}
		90C26, 68T07, 65L05, 65L20, 39A30
	\end{amscode}

\section{Introduction}	
\label{intro}

\paragraph*{Families of optimizers}
~~\\
Gradient Descent (GD) and its mini-batch version (SGD) are among the most popular algorithms to train neural networks nowadays. GD updates the parameters of the model $\theta_n$ in the opposite direction of the gradient with a small learning rate/step size $\eta>0$ to minimize a cost function $\R$:
\begin{equation}
	\theta_{n+1}=\theta_n-\eta\nR(\theta_n).
	\label{GD_update}
\end{equation}
This procedure suffers from two main drawbacks that have given birth to two new families of optimizers:
\begin{itemize}
	\item inertial algorithms. For $\mu$-strongly convex function \cite{Polyak, Nesterov_book}, the rate of convergence is suboptimal. This is why Polyak \cite{Polyak} suggested a new kind of algorithm known as Inertial optimizers. The idea is to use past gradients to update the parameters. One of the most popular inertial algorithm is Momentum that introduces a new intermediate sequence $(v_n)_{n\in \mathbb{N}}$ representing an exponential moving average of the gradients ($v_0=0$):
	\begin{equation}
		\left\{
		\begin{array}{ll}
			v_{n+1} = \beta v_n + \eta \nR(\theta_n),  \\
			\theta_{n+1} = \theta_n - \eta v_{n+1},
			\label{Momentum_update}
		\end{array}
		\right.
	\end{equation}
	where $0<\beta<1$ is a new hyperparameter measuring the trade-off between actual gradient and previous ones. 
	\item Adaptive GD algorithms. Since GD is hightly sensitive to the choice of $\eta$, it needs a fine tuning to achieve good performances. In this perspective, variants of GD called adaptive GD have become particularly popular because they adjust the learning rate using only the square coordinates of the past gradients: see AdaDelta \cite{Adadelta}, Adam \cite{Adam} and AdaGrad \cite{Adagrad}. The prototype of all these methods is well illustrated by the RMSProp procedure \cite{RMSProp}:
	\begin{equation}
		\left\{
		\begin{array}{ll}
			s_{n+1} = \beta s_n +(1-\beta)\nR(\theta_n )^2,  \\
			\theta_{n+1} = \theta_n - \eta \dfrac{\nR(\theta_n)}{\epsilon_a+\sqrt{s_{n+1}}}, 
		\end{array}
		\right.
		\label{RMSProp_update}
	\end{equation}
	where $\epsilon_a>0$ is usually a small parameter preventing division  by zero. Let us recall that the square and the division have to be understood as component-wise operations.   
\end{itemize}
The vast majority of ML optimizers (like Adam: default choice in many Machine Learning (ML) frameworks) may be seen as a mixture of Polyak/Nesterov acceleration technique and RMSProp scaling. This is why this work focuses on Momentum and RMSProp algorithms. Two main problems will be in the spotlight: the lack of guarantees for the convergence of iterates especially in the non-convex setting and the tuning of hyperparameters $\eta$, $\beta$ and $\epsilon_a$. 

\paragraph*{Momentum review}
~~\\
The theoretical guarantees for Momentum can be organized according to the hypothesis on $\R$:
\begin{itemize}
	\item If $\R$ is $\mu$-strongly convex and $L$-smooth, the function value $\left(\R(\theta_n)-\R(\theta^*)\right)_{n\in \mathbb{N}}$ ($\theta^*$ the global minimum of $\R$) converges as $(1-\sqrt{\mu/L})^n$ \cite{Polyak, Siegel} where GD converges as $(1-\mu/L)^n$ explaining the acceleration terminology. To obtain this improvement, one has to choose $\eta$ and $\beta$ given by a formulae depending on $L$ and $\mu$ that are unknown for a given problem.  
	\item If $\R$ is only convex and $L$-smooth, it is shown in \cite{Ghadimi} that the Cesaro average of the iterates converges. By adding the Polyak-Lojasiewicz (PL) assumption and the uniqueness of the minimum, F. Aujol \cite{Aujol_recap} retrieves the same acceleration rate as the strongly convex case for function values. By some restart techniques, it is possible to get rid of the unicity assumption \cite{aujol_non_unicity} and the knowledge of smoothness parameters \cite{aujol_free_fista}. 
	\item For non-convex functions if $\R$ is lower bounded and $L$-smooth, then the gradients converge to 0 \cite{momentum_gradient}. Like the GD optimizer \cite{Absil} it is needed to consider an additional hypothesis to expect convergence of the iterates, namely the Lojasiewicz or Kurdyka-Lojasiewicz (KL) property \cite{Lojasiewicz_gradient, Kurdyka}. This assumption is absolutely not restrictive for neural networks as it takes into account analytical, semi-analytical, semi-algebraic and more generally o-minimal functions \cite{Bolte_semi_analytic}. By assuming bounded iterates (or $\R$ coercive) and the \textbf{global $L$-smooth} condition, the authors in \cite{iPiano, diff_convexe, KL_taux} prove convergence of the iterates to a critical point of $\R$ and establish "abstract" convergence rates \textbf{without computing their dependence on the hyperparameters and the Lojasiewicz coefficients} (to compare it with GD like in the strongly convex case). In \cite{algebraic_momentum}, the coercivity and the global smoothness are replaced by the assumption that the continuous-time gradient trajectories of $\R$ are bounded, which is not easily verified by a general neural network.
	\item To be exhaustive, one can mention the recent analysis of the differential equation associated to Momentum \cite{HB_nonconvex_acceleration} under the PL condition. The convergence rate is explicit and reveals the same acceleration than in the convex case but $\beta$ should depend on $L$ and the PL coefficient. Other works with assumptions stronger than PL \cite{Momentum_averaged_out} or specific to relu networks \cite{Momentum_wide_relu,Nesterov_deep_relu} are valuable.  
\end{itemize}
However, all \textbf{the results above can not be used in practice} since they are established for $\eta$ and $\beta$ at best depending on $L$ (or we just know that it exists a good choice). Indeed, for neural networks, $L$ is not known.     

\paragraph*{RMSProp results}
~~\\
For RMSProp (and adaptive GD optimizers) much less is known:
\begin{itemize}
	\item For strongly convex functions, assuming the gradient and the parameter sequence bounded, it is possible to obtain convergence for the Cesaro average of the function value \cite{Adam, rms_strong_convex} with a rate of $\mathcal{O}(1/\sqrt{n})$ for learning rates that decrease as $1/\sqrt{n}$.
	\item In the non-convex setting, in addition to \textbf{global $L$-smoothness}, the authors either assume that the sequence $(\theta_n)_{n\in \mathbb{N}}$ is bounded (or apply a projection at each step) or that the \textbf{gradient is bounded} to establish a gradient convergence \cite{zou_rms, weak_rms, rms_penalty} for constant step size. The bounded assumption is removed in \cite{rms_not_bounded} using a learning rate of the form $\eta_n=1/\sqrt{n}$. To the best of our knowledge, the only work concerning convergence of the iterates is \cite{rms_bianchi} for KL functions but the authors assume the \textbf{boundedness of the moving average} $(s_n)_{n\in \mathbb{N}}$. Such an hypothesis is not easily satisfied a priori.  
\end{itemize}
As in the case of Momentum, all these results set the values of $\eta$ and $\beta$ depending on $L$. 

\paragraph*{Tuning hyperparameters}
~~\\
Given that the smoothness constant ($L$) is not available, practitioners have to tune the value of the different hyperparameters \cite{gridsearch_adam}. One may think that adaptive GD are less sensitive to these choices than GD and Momentum as the moving average $(s_n)$ is used to scale the learning rate. 
In \cite{rms_tuning}, Wilson showed that adaptive GD methods need the same amount of tuning as GD to obtain satisfactory performances. To tackle the expensive tuning problem, some strategies have been proposed for Momentum (at the best of our knowledge, there is no attempt for RMSProp). In \cite{adaptive_HB}, for a $\mu$-strongly convex function, they build an online approximation of $L$ and $\mu$ to compute the optimal Polyak hyperparameters \cite{Polyak}. But this approach seems very related to the convex assumption. Another attempt is based on the GD Armijo rule \cite{armijo, wolfe1, wolfe2} popular in the optimization community to adaptively tune $\eta$. The learning rate $\eta_n$ is chosen by a linesearch in order to achieve a sufficient decrease of the function $\R$:
\begin{equation}
	\R(\theta_n-\eta_n \nR(\theta_n))-\R(\theta_n) \leq -\lambda \eta_n \|\nR(\theta_n)\|^2,
	\label{armijo_rule}
\end{equation}
where $\lambda \in ]0,1[$. 
\cite{armijo_momentum} suggest a way to use the Armijo backtracking for Momentum. A first linesearch is done on $\beta$ to impose the cosinus between $v_n$ and $\nR(\theta_n)$ to be positive. Then a second linesearch is performed on the following Armijo rule:
\begin{equation}
	\R(\theta_n-\eta_n v_n)-\R(\theta_n) \leq -\lambda \eta_n \nR(\theta_n)\cdot v_n.
	\label{Tong_armijo}
\end{equation}
This strategy has two drawbacks. First, the authors need to introduce four new hyperparameters not related to the linesearches. Their (physical) meaning and influence on the convergence of the algorithm are not discussed. In addition to this, condition \eqref{Tong_armijo} combined with the first linesearch forces $\R$ to decrease. This changes completely the Momentum dynamics characterizing by the non-monotony of $\R$ \cite{non_monotone_HB, attouch_continuous}. \\
It is also usual in the litterature to proceed to a linesearch based on the descent lemma (on $\R$) in order to estimate $L$. Some guarantees on the gradient convergence are known for exact linesearch \cite{Nesterov_dual} or inexact linesearch if the time step is assumed bounded by below by some positive constant \cite{Fista_goldstein}. However, these approaches do not decrease the mechanical energy and do not ensure any stability properties \cite{Bilel, Lyap_Theory_Bilel}.

This review of the literature raises an important question: \textbf{is it possible to build an adaptive hyperparameter strategy both for Momentum and RMSProp ensuring convergence of the iterates under the KL-property ?}

\paragraph*{Our Results}
~~\\
Our main contributions are:
\begin{enumerate}
	\item We propose \textbf{adaptive time step strategies} both for Momentum and RMSProp that preserve dynamical properties of these systems and introduce few hyperparameters (often related to linesearch in our case and not to the dynamics) in section \ref{generalize_armijo}. Their role and influence will be completely clarified through the paper.
	\item We prove convergence of the iterates for this adaptive Momentum under Lojasiewicz assumption (it works also for KL but we consider Lojasiewicz function to simplify the notations) removing the global $L$-smoothness hypothesis in theorem \ref{momentum_th} of section \ref{analysis_momentum}. We compute \textbf{explicit convergence rates} that can be compared to the GD ones. As a corollary \ref{momentum_corollary}, we obtain a result for the constant step size Momentum. 
	\item In section \ref{analysis_rms}, we establish \textbf{convergence of the iterates} for adaptive RMSProp for Lojasiewicz function \textbf{removing the global $L$-smoothness, the bounded gradient and the bounded moving average assumption} and get convergence rates too, in theorem \ref{rms_th}. As a corollary \ref{rms_corollary}, we obtain the first strong result for constant step size RMSProp assuming only Lojasiewicz inequality and $L$-smoothness. Our analysis highlight the significant role of $\epsilon_a$ in the dynamics.
	\item In the last section \ref{num_experiments}, constant step size and adaptive strategies are numerically evaluated on two ML benchmarks.
\end{enumerate}
			
\section{Generalizing Armijo by Lyapunov Theory}
\label{generalize_armijo}

In this section, we present the well-known ODE framework to analyse iterative algorithms and use it to build two adaptive time step strategies for Momentum and RMSProp.

\paragraph*{ODE Framework}
~~\\
In these recent years, many researchers \cite{variational_perspective, continuous_general} have underlined the deep relations between ODE and optimization and use them to understand the optimizers' behaviour.  For instance, in order to analyse the sequence produced by GD and compare it to Momentum, the update rule \eqref{GD_update} is seen as an Euler discretization of the gradient flow (GF) \cite{Polyak, variational_perspective} when $\eta$ is the time step of the numerical method:
\begin{equation}
	\theta'(t)=-\nR(\theta(t)).
	\label{GF}
\end{equation}
In fact, we may expect that when $\eta$ is small enough, the dynamic of GF will describe the behaviour of GD. The same can be done for Momentum. By introducing an affine relation between $\beta$ and $\eta$, namely $\beta = 1-\bar{\beta}\eta$ ($\bar{\beta}>0$), the relations \eqref{Momentum_update} can be written using only the parameter sequence:
\begin{equation*}
	\dfrac{\theta_{n+1}-2\theta_n+\theta_{n-1}}{\eta^2} + \bar{\beta}\dfrac{\theta_n-\theta_{n-1}}{\eta}+\nR(\theta_n)=0,
\end{equation*}
which may be seen as finite differences for the ODE:
\begin{equation}
	\theta''(t)+\bar{\beta}\theta'(t)+\nR(\theta(t))=0,
	\label{Momentum_ode_theta}
\end{equation}
when $\eta$ tends to 0; see \cite{variational_perspective, continuous_general, Bilel} for more details. Let us rewrite \eqref{Momentum_ode_theta} as a first order ODE to reintroduce $v$:
\begin{equation}
	\left\{
	\begin{array}{ll}
		\theta'(t) = -v(t), \\ 
		v'(t) = \bar{\beta}v(t) + \nR(\theta(t))^2. 
	\end{array}
	\right.
	\label{Momentum_ode}
\end{equation}
The dynamic \eqref{Momentum_ode} is very different of GF. In fact, for GF, $\R$ decreases monotonically with \textbf{a dissipation of $-\|\nR(\theta(t))\|^2$}, that is to say for $\theta(t)$ solution of \eqref{GF} we have:
\begin{equation}
	\dfrac{d\R(\theta(t)}{dt} = -\|\nR(\theta(t))\|^2 \leq 0.
	\label{dissipation_GF}
\end{equation}
However, for $(\theta(t),v(t))$ solution of \eqref{Momentum_ode}, we have:
\begin{equation*}
	\dfrac{d\R(\theta(t))}{dt} = -\nR(\theta(t))\cdot v(t).
	\label{R_Momentum}
\end{equation*}
The function itself may increase but not the mechanical energy \cite{non_monotone_HB} defined by:
\begin{equation}
	V(\theta,v) = \R(\theta)+\frac{\|v\|^2}{2}.
	\label{Lyapunov_momentum}
\end{equation}
Indeed, a simple calculation gives (\textbf{dissipation of $-\bar{\beta} \|v(t)\|^2$}):
\begin{equation}
	\dfrac{dV(\theta(t),v(t))}{dt} = -\bar{\beta}\|v(t)\|^2 \leq 0.
	\label{dissipation_momentum}
\end{equation}
The functions $\R$ for GF \eqref{GF} and $V$ for \eqref{Momentum_ode} are called Lyapunov functions as they decrease along the trajectories of the ODE. Their monotony play a central role in the stability of dynamical systems, see \cite{Khalil} for an overview (bad things can happen for unstable systems, see \cite{Bilel}). In addition, they are the key tools to prove asymptotic convergence of $(\theta(t))_{t\geq 0}$ for \eqref{GF} and \eqref{Momentum_ode} in the Lojasiewicz setting \cite{Haraux_autonomous, attouch_continuous}. 

\paragraph*{Discretization}
~~\\
Given that \eqref{GD_update} and \eqref{Momentum_update} with $\beta=1-\bar{\beta}\eta$ are consistent discretizations of \eqref{GF} and \eqref{Momentum_ode}, we have respectively for $\eta$ small enough:
\begin{equation*}
	\R(\theta_{n+1})-\R(\theta_n) = -\eta \|\nR(\theta_n)\|^2+o(\eta),
\end{equation*}
\begin{equation}
	V(\theta_{n+1},v_{n+1})-V(\theta_n,v_n) = -\eta\|v_{n+1}\|^2+o(\eta).
	\label{Taylor_momentum}
\end{equation}
As explained in \cite{Bilel}, common choices of hyperparameters may violate the decreasing property (and the decreasing rate) of both $\R$ and $V$ given by the ODE. This produces unstable trajectories that have a significant impact on the optimization performances and the training time. With this point of view, the classical Armijo rule \eqref{armijo_rule} can be seen as a way to force these (nice) continuous properties up to a constant $\lambda$. Indeed, the authors in \cite{Absil} prove the stability and convergence of the trajectories produced by the Armijo backtracking. In few words, our strategy to generalize Armijo to Momentum and RMSProp could be sum up in three points:
\begin{enumerate}
	\item An ODE has to be identified in such a way that the optimizer is a consistent discretization of this system. To do this, one can rely on the literature on this subject, see \cite{continuous_general} for a general ODE that takes into account all the ML optimizers.
	\item Then a Lyapunov function for this continuous equation has to be built (it is not unique) to get a relation of the form \eqref{dissipation_momentum}.
	\item Finally, we adapt the Armijo rule \eqref{armijo_rule} by replacing $\R$ by the Lyapunov function and the right hand side by the dissipation obtained in step 2. This is why we call the Armijo rule, the Lyapunov inequality.
\end{enumerate}

\paragraph*{Adaptive Momentum}
~~\\
Concretely, this gives the following update for Momentum:
\begin{equation}
	\left\{
	\begin{array}{ll}
		v_{n+1} = \beta_n v_n + \eta_n \nR(\theta_n),  \\
		\theta_{n+1} = \theta_n - \eta_n v_{n+1}.
		\label{Momentum_eq}
	\end{array}
	\right.
\end{equation}
where $\beta_n = 1-\bar{\beta}\eta_n$ with the following conditions on the time step (Lyapunov inequality):
\begin{equation}
	V(\theta_{n+1},v_{n+1})-V(\theta_n,v_n) \leq -\lambda \bar{\beta} \eta_n \|v_{n+1}\|^2,
	\label{Momentum_condition}
\end{equation}
and $\eta_n \leq \frac{1}{\bar{\beta}}$. This restriction on the learning rate is introduced in order for $\beta_n$ to be positive (necessary to compute effectively a moving average). To be complete it remains to explain how we proceed for the linesearch that imposes \eqref{Momentum_condition}:
\begin{enumerate}
	\item For the first iteration, the time step begins with the value $\eta_{init}=\frac{1}{\bar{\beta}}$.
	\item At each iteration ($n\geq 1)$, we begin the linesearch with the value $\eta_n^0 = \min\left({f_2\eta_{n-1}, 1/\bar{\beta}}\right)$ where $\eta_{n-1}$ is the admissible time step of the previous iteration and $f_2>1$. Contrary to classical backtracking \cite{armijo}, we decide to begin the linesearch not on a constant time step but to use the previous one up to factor $f_2$. This technique was first introduced in \cite{Rondepierre,Lyap_Theory_Bilel} and is a lot more efficient than the classical one.
	\item Then (as in classical linesearch) we reduce the initial time step $\eta_n^0$ by a factor $f_1>1$, $\eta_n^i = \eta_n^0/f_1^i$ until equation \eqref{Momentum_condition} is satisfied. This process finishes according to the Taylor expansion \eqref{Taylor_momentum}.
\end{enumerate}
The optimizer is written as algorithm \ref{algo_adaptive_momentum} in appendix \ref{appendix_algorithms}. The system \eqref{Momentum_eq} admits by construction a Lyapunov function which gives it nice stability and attractivity properties, see \cite{Lyap_Theory_Bilel,Lyapunov_non_monotonic, Lyapunov_discrete}.

\paragraph*{Adaptive RMSProp}
~~\\
The same could be done for RMSProp. By also imposing the relation $\beta = 1-\bar{\beta}\eta$, the following ODE can be derived as an approximation of \eqref{RMSProp_update}, see \cite{continuous_general}:
\begin{equation}
	\left\{
	\begin{array}{ll}
		\theta'(t) = - \dfrac{\nR(\theta(t))}{\epsilon_a+\sqrt{s(t)}}, \\
		s'(t) = \bar{\beta}\left[-s(t) + \nR(\theta(t))^2\right].
	\end{array}
	\right.
	\label{rms_ode}
\end{equation}
We can exhibit $\R$ as a Lyapunov function for \eqref{rms_ode}:
\begin{equation*}
	\dfrac{d\R(\theta(t))}{dt} = - \left|\left|\dfrac{\nR(\theta(t))}{\sqrt{\epsilon_a+\sqrt{s(t)}}}\right|\right|^2 \leq 0.
\end{equation*}
As a result, the update rule for RMSProp is the following:
\begin{equation}
	\left\{
	\begin{array}{ll}
		s_{n+1} = \beta_n s_n +(1-\beta_n)\nR(\theta_n)^2,  \\
		\theta_{n+1} = \theta_n - \eta_n \dfrac{\nR(\theta_n)}{\epsilon_a+\sqrt{s_{n+1}}},
	\end{array}
	\right.
	\label{RMSProp_eq}
\end{equation}
with $\beta_n = 1-\bar{\beta}\eta_n$. The Armijo rule becomes (Lyapunov inequality):
\begin{equation}
	\R(\theta_{n+1})-\R(\theta_n) \leq - \lambda \eta_n \left|\left|\dfrac{\nR(\theta_n)}{\sqrt{\epsilon_a+\sqrt{s_{n+1}}}}\right|\right|^2,
	\label{RMSProp_Lyapunov}
\end{equation}
with the supplementary condition $\eta_n \leq 1/\bar{\beta}$. The linesearch procedure is exactly the same as Momentum. The procedure is described as algorithm \ref{algo_adaptive_rms} in appendix \ref{appendix_algorithms}.\\ 

The inequalities \eqref{Momentum_condition} and \eqref{RMSProp_Lyapunov} transform systems \eqref{Momentum_eq} and \eqref{RMSProp_eq} into dissipative systems, see \cite{Haraux_autonomous}. This kind of dynamical systems enjoys asymptotic convergence for a large class of functions (and other proprieties). The next sections will rigorously establish this fact.

\section{Analysis of Adaptive Momentum}
\label{analysis_momentum}

Before stating the convergence result for Adaptive Momentum, let us recall some classical hypotheses in non-convex optimization and the definition of a Lojasiewicz function \cite{Lojasiewicz_gradient, Kurdyka}.

\begin{assumption}
	$\R$ is differentiable on its whole domain $\mathbb{R}^N$ and lower bounded. 
	\label{c1}
\end{assumption}

\begin{assumption}
	$\nR$ is locally Lipshitz continuous.
	\label{local_lip}
\end{assumption}

\begin{assumption}
	$\nR$ is globally Lipshitz continuous or equivalently $L$-smooth.
	\label{global_lip}
\end{assumption}

\begin{assumption}
	$\R$ is coercive that is to say $\R(\theta)\to +\infty$ when $\|\theta\| \to +\infty$.
	\label{coercive}
\end{assumption}

\begin{assumption}
	$\bar{\beta}>1$ and $\lambda < \frac{1}{2\bar{\beta}}$.
	\label{coeffs}
\end{assumption}

\begin{definition}[Lojasiewicz]
	We say that a differentiable function $f: \Rb^m \mapsto \Rb$ satisfies Lojasiewicz inequality at a critical point $x^*\in\Rb^m$ if there are $c>0$, $\alpha \in ]0,1[$ and a neighborhood $\mathcal{U}_{x^*}$ of $x^*$ such that:
	\begin{equation}
		x\in \mathcal{U}_{x^*} \Rightarrow \|\nabla f(x)\| \geq c \|f(x)-f(x^*)\|^{1-\alpha}.
		\label{loj_condition}
	\end{equation}
	The constants $\alpha$ and $c$ are called the Lojasiewicz coefficients associated to $x^*$. Sometimes when the context is clear, we will just use $\mathcal{U}$ in place of $\mathcal{U}_{x^*}$. We say that $f$ is a Lojasiewicz function if it satisfies the Lojasiwicz inequality in the neighborhood of all its critical points. 
\end{definition}

Let us now state the main result of this section.

\begin{theorem}
	Assume \ref{c1}, \ref{local_lip}, \ref{coercive}, \ref{coeffs}, and that $V$ is Lojasiewicz. Then the sequence $(\theta_n, v_n)$ produced by Adaptive Momentum converges to $(\theta^*,0)$ with $\theta^*$ a critical point of $\R$.
	Denoting by $\alpha$ and $c$ the Lojasiewicz coefficients associated to $\theta^*$, the following convergence rates hold:
	\begin{itemize}
		\item if $\alpha=\frac{1}{2}$, we have:
		\begin{equation*}
			\|\theta_n-\theta^*\| = \mathcal{O}\left(\exp{\left(-\frac{1}{4}\displaystyle{\sum_{k=1}^{n-1}\frac{1}{B_k}}\right)}\right).
		\end{equation*}
		\item If $0<\alpha<\frac{1}{2}$v, we have:
		\begin{equation*}
			\|\theta_n-\theta^*\| = \mathcal{O}\left(\left[\displaystyle{\sum_{k=1}^{n-1}\frac{1}{B_k}}\right]^{-\frac{\alpha}{1-2\alpha}}\right).
		\end{equation*}
		\item If $\frac{1}{2}<\alpha<1$, then $(\theta_n)_{n\in \mathbb{N}}$ converges in finite time, that is to say it exists $n'\geq 0$ such that $\theta_n = \theta^*$ for $n\geq n'$. 
	\end{itemize}
	In the above expressions, we have:
	\begin{equation}
		B_k \coloneqq \frac{2}{\lambda \bar{\beta}c^2\eta_k^3}\max{\left(1,r_k\left[
			(1-\bar{\beta})\eta_k+1\right]^2\right)}
		\label{expr_Bn}
	\end{equation}
	with $r_k\coloneqq\dfrac{\eta_k}{\eta_{k-1}}$.
	\label{momentum_th}
\end{theorem}
\begin{proof}
	See Appendix \ref{appendix_momentum}.
\end{proof}

The constant step size Momentum analysis appears as a consequence of this theorem:

\begin{corollary}
	Assume \ref{c1}, \ref{global_lip}, \ref{coercive} and that $V$ is Lojasiewicz. It exists $\eta_o>0$ such that the sequence $(\theta_n, v_n)$ produced by Momentum \eqref{Momentum_update} (with $\beta=1-\bar{\beta}\eta$) converges to $(\theta^*,0)$ where $\theta^*$ is a critical point of $\R$, for any choice of $\bar{\beta}\geq 1$.
	Denoting by $\alpha$ and $c$ the Lojasiewicz coefficients associated to $\theta^*$, the following convergence rates hold:
	\begin{itemize}
		\item if $\alpha=\frac{1}{2}$, we have:
		\begin{equation}
			\|\theta_n-\theta^*\| = \mathcal{O}\left(\exp{\left(-\frac{n}{4B}\right)}\right).
			\label{momentum_exp_cst}
		\end{equation}
		\item If $0<\alpha<\frac{1}{2}$, we have:
		\begin{equation*}
			\|\theta_n-\theta^*\| = \mathcal{O}\left(n^{-\frac{\alpha}{1-2\alpha}}\right).
		\end{equation*}
		\item If $\frac{1}{2}<\alpha<1$, then $(\theta_n)_{n\in \mathbb{N}}$ converges in finite time.
	\end{itemize}
	In the above expression, we have:
	\begin{equation*}
		B \coloneqq \frac{2}{\lambda \bar{\beta}c^2\eta_o^3}\max{\left(1,\left[
			(1-\bar{\beta})\eta_o+1\right]^2\right)}.
	\end{equation*}
	\label{momentum_corollary}
\end{corollary}
\begin{proof}
	See Appendix \ref{appendix_momentum}.
\end{proof}

Let us give a sketch of the proof through some lemmas and at the same time comment the results.
The first lemma will justify that if the algorithm converges, $(v_n)_{n\in \mathbb{N}}$ converges to $0$. To understand that it is not obvious, let us imagine (just in this paragraph) that $z_n\coloneqq (\theta_n,v_n)$ converges. By passing to the limit in the inequality \eqref{Momentum_condition} and by continuity of $V$, we get immediately:
\begin{equation*}
	\displaystyle{\lim_{n\to+\infty}} \eta_n \|v_{n+1}\|^2 = 0. 
\end{equation*}
If the time step sequence $(\eta_n)$ is constant, then it is straightforward to deduce $v_n \to 0$. However, in the adaptive time step scenario, it can happen that $\eta_n$ converges to 0 and $v_n$ not. To avoid this situation, we need to prove $\inf \eta_n>0$. To do this, it is sufficient to show that the condition \eqref{Momentum_condition} can be achieved with a fixed learning rate. This is the topic of the next lemma.

For $\theta,v \in \Rb^N$ and a fix $\eta$, let us introduce the notations $v_{\eta}\coloneqq (1-\bar{\beta}\eta)v+\eta \nR(\theta)$ and $\theta_{\eta}=\theta-\eta v_{\eta}$.
\begin{lemma}
	Assume \ref{c1}, \ref{local_lip}, \ref{coeffs} and consider a compact set $K\subset \Rb^N$. There is $0<\eta_o\leq \frac{1}{\bar{\beta}}$, depending only on $\lambda$, $\bar{\beta}$ and the smoothness constant of $\R$ on $K$, such that for all $\eta\leq \eta_o$ and $\theta \in K$ we have:
	\begin{equation*}
		V(\theta_{\eta},v_{\eta})-V(\theta,v) \leq -\lambda \bar{\beta} \eta \|v_{\eta}\|^2.
	\end{equation*}
	\label{Momentum_eta_inf}
\end{lemma}

\begin{proof}
	Define $f(\theta, v) \coloneqq V(\theta_{\eta},v_{\eta})-V(\theta,v)+\lambda \bar{\beta}\|v_{\eta}\|^2$ for $\eta \leq \frac{1}{\bar{\beta}}$.
	$\nR$ is locally Lipshitz so it is globally Lipshitz on $K$. Denoting by $L=L_K$ its smoothness constant, we can write:
	\begin{equation*}
		\R(y)-\R(x) \leq \nR(x)\cdot (y-x) + \frac{L}{2}\|y-x\|^2.
	\end{equation*}
	Applying this to $y=\theta_{\eta}$ and $x=\theta$, this leads to:
	\begin{equation*}
		\R(\theta_\eta)-\R(\theta) \leq -\eta \nR(\theta)\cdot v_{\eta}+\frac{L\eta^2}{2}\|v_{\eta}\|^2.
	\end{equation*}
	Rearranging the terms we have:
	\begin{multline*}
		f(\theta,v) \leq \eta \bigg[\frac{\bar{\beta}^2 L}{2}\eta^3+\bar{\beta}(\lambda\bar{\beta}^2-L)\eta^2+
		\frac{1}{2}(\bar{\beta}^2+L-4\lambda\bar{\beta}^2)\eta+\bar{\beta}(\lambda-1)\bigg]\|v\|^2 \\ + \eta^2 \left[-\bar{\beta}L\eta^2+(L-2\lambda\bar{\beta}^2)\eta+2\lambda \bar{\beta}\right]v \cdot \nR(\theta) + \frac{\eta^2}{2}\left[L\eta^2+2\lambda\bar{\beta}\eta-1\right]\|\nR(\theta)\|^2.
	\end{multline*}
	As the polynomial in front of $v\cdot \nR(\theta)$ is positive since $\eta \leq \frac{1}{\bar{\beta}}$, using $v\cdot \nR(\theta) \leq \frac{\|v\|^2+\|\nR(\theta)\|^2}{2}$, we can write:
	\begin{equation}
		f(\theta,v) \leq \eta P(\eta) \|v_n\|^2 + \eta^2 Q(\eta) \|\nR(\theta)\|^2,
		\label{f_polys}
	\end{equation}
	where:
	\begin{multline*}
			P(\eta) = \frac{\bar{\beta}L}{2}(\bar{\beta}-1)\eta^3+\left[\lambda \bar{\beta}^2(\bar{\beta}-1)+\frac{L}{2}(1-2\bar{\beta})\right]\eta^2
			+\frac{1}{2}\left[\bar{\beta}^2+L+2\lambda\bar{\beta}-4\lambda\bar{\beta}^2\right]\eta \\
			+ \bar{\beta}(\lambda-1),
	\end{multline*}
	\begin{equation*}
			Q(\eta) = \frac{L}{2}(1-\bar{\beta})\eta^2+\frac{1}{2}\left[L+2\lambda\bar{\beta}-2\lambda\bar{\beta}^2\right]\eta
			+ \lambda\bar{\beta}-\frac{1}{2}.
	\end{equation*}
	To finish the proof it is sufficient to find a domain $]0,\eta_o]$ where both $P$ and $Q$ are negative.
	By condition \ref{coeffs}, $Q(0) = \lambda \bar{\beta}-\frac{1}{2}<0$. So there exists $\eta_Q>0$ such that $Q$ is negative on $]0,\eta_Q]$ by continuity of $Q$. As $P(0)=\bar{\beta}(\lambda-1)<0$, we also obtain $\eta_P>0$ such that $P$ is negative on $]0,\eta_P]$. We set $\eta_o=\min{\left(\eta_P,\eta_Q, 1/\bar{\beta}\right)}$ to finish the proof.
\end{proof}

\begin{remark}
	The previous analysis enables to have a estimate of $\eta_o$ depending on $L$ but the sign analysis of the polynomials are tedious (comparing the roots of $P$ and $Q$). 
	Let us compute an estimation of $\eta_o$ for two values of $\bar{\beta}$:
	\begin{itemize}
		\item Fortunately for the value $\bar{\beta}=1$ (and $\lambda\leq 0.5$), $P$ is of degree 2 and $Q$ of degree $1$
		\begin{equation*}
			Q(\eta) \leq 0 \Leftrightarrow \eta \leq \eta_{root} \coloneqq \dfrac{1-2\lambda}{L}.
		\end{equation*}
		Finally for $0<\eta \leq \eta_{root}$, $P(\eta) \leq P(\eta_{root})=0$. Consequently, inequality \eqref{f_polys} gives that $\eta \leq \eta_o \coloneqq \min{\left(\frac{1}{\bar{\beta}}, \eta_{root}\right)} \implies f(\theta_n,v_n)\leq 0$.
		\item For $\bar{\beta}=2$ and $\lambda < \frac{1}{4}$, a precise analysis of the variations of $P$ and $Q$ gives $\eta_o$. Define:
		\begin{equation*}
			\begin{array}{ll}
				L_1 = 2(1-2\lambda)-2\sqrt{1-4\lambda}, \\
				L_2 = 2(1-2\lambda)+2\sqrt{1-4\lambda}.
			\end{array}
		\end{equation*}
		If $L_1 < L < L_2$, $\eta_{root} \coloneqq \frac{1}{\bar{\beta}}$. Otherwise, $\eta_{root}$ is defined as:
		\begin{equation*}
			\eta_{root} \coloneqq \frac{1}{2}-\frac{2\lambda}{L}-\sqrt{\frac{1}{4}+\frac{2\lambda-1}{L}+\frac{4\lambda^2}{L^2}} \underset{L\to +\infty}{\sim} \frac{1}{L}.
		\end{equation*}
		Then $\eta_o = \min{\left(\eta_{root}, 1/\bar{\beta}\right)}$.
	\end{itemize}
	\label{momentum_eta_ex}
\end{remark}

\begin{remark}
	According to the analysis above, the learning rate for Momentum seems to be of the same order as for GD that is to say the inverse of $L$. Let us recall that the convergence rate for GD when $\alpha=0.5$ is typically of the form $e^{-\eta c^2 n} \approx e^{-L^{-1} c^2 n}$\cite{Absil} whereas in our case it is $e^{-\eta^3 c^2 n}\approx e^{-L^{-3} c^2 n}$ for Momentum. Two comments should be made about this difference:
	\begin{itemize}
		\item The $c$-dependence of the convergence rate \eqref{momentum_exp_cst} is the same as GD. The constant $c^2$ is generally the smallest eigenvalue of the hessian. In other words it plays the same role as the constant $\mu$ for $\mu$-strongly convex functions. Our analysis does not succeed in catching the improvement $\sqrt{\mu}=c$ for strongly convex functions \cite{Siegel}. But even for convex functions with the Polyak-Lojasiewicz assumption (global Lojasiewicz with $\alpha=0.5$), the question is open if there is no uniqueness of the minimum, see \cite{Aujol_recap}.
		\item The $\eta$ or $L$-dependence of the convergence rate \eqref{momentum_exp_cst} is completely different. It seems that Momentum in our analysis drastically slows down in stiff regions ($L\gg 1$) and it is independent of the value $L$ in flat region ($L$ small enough in order for $\eta$ to be close to $1/\bar{\beta}$). 
	\end{itemize}
\end{remark}

After proving that the learning rates can not vanish, we need to show that if the iterates enter a neighborhood $\mathcal{U}$, when the Lojasiewicz inequality is valid, they can not escape from $\mathcal{U}$. To do this, we will establish an inequality that compares the distance between two successive parameters $\Delta_n \coloneqq \|\theta_{n+1}-\theta_n\|$ and $V(z_n)$ (that decreases), in such a neighborhood. 

\begin{lemma}
	Assume that $V$ is Lojasiewicz. Let us consider $z^*\coloneqq(\theta^*,0)$ where $\theta^*$ is a critical point of $\R$ and assume that $z_n \in \mathcal{U}_{z^*}$ for some $n\geq 0$. We have:
	\begin{equation}
		\Delta_{n+1} \leq \frac{1}{2}(\Delta_n-\Delta_{n+1})+
		\dfrac{2A_n}{\lambda\alpha c \bar{\beta} \eta_{n-1} \eta_n}\left[V(z_n)^{\alpha}-V(z_{n+1})^{\alpha}\right].
		\label{delta_ineq}
	\end{equation}
	where:
	\begin{equation*}
		A_n \coloneqq \max{\left(\eta_{n-1}, \eta_n(\eta_n+\beta_n)\right)},
	\end{equation*}
	and $\alpha$, $c$ are the Lojasiewicz coefficients associated to $z^*$.
	\label{lemma_delta_momentum}
\end{lemma}
\begin{proof}
	See Appendix \ref{appendix_momentum}.
\end{proof}

This lemma with the dissipation inequality \eqref{Momentum_condition} will make it possible to compare $\|\nabla V(z_n)\|$ and $\left[V(z_{n+1})-V(z_{n-1})\right]$ like in \cite{Absil, iPiano}. Thanks to the Lojasiewicz inequality, we derive a relation between $\left[V(z_{n+1})-V(z_{n-1})\right]$ and $V(z_n)^{\alpha}$. To deduce a bound for $V(z_n)$ from it, a two-step non linear Gronwall lemma \ref{Gronwall_two_step} is presented in the appendix \ref{appendix_momentum}.
Combining these lemmas, the complete proof of theorem \ref{momentum_th} is written in appendix \ref{appendix_momentum}.

\section{Analysis of Adaptive RMSProp}
\label{analysis_rms}

In this section, we present the convergence results for Adaptive RMSprop and constant step size RMSProp. The proof follows the same steps as the one for Momentum. We finally highlight the main differences with GD that this new analysis reveals.
Given that the analysis and the computations for RMSProp are more involved, let us introduce some notations and sequences to write the result in a compact form. 

Let us note that under assumption \ref{coercive}, the sequence $(\theta_n)_{n\in \mathbb{N}}$ produced by \eqref{RMSProp_eq} is bounded. Indeed if it was not the case, $\R(\theta_n) \to +\infty$ because of coercivity. But this is in contradiction with the decreasing of $\R$ \eqref{RMSProp_Lyapunov}. So under assumption \ref{local_lip}, it is possible to consider the lipshitz constant of $\nR$ on the compact $K=\overline{\{\theta_n, n\geq 0\}}$. For $n\geq 0$, $j\leq n$, let us define:
\begin{equation*}
	B_{n-j} \coloneqq (1-\beta_{n-j})\displaystyle{\prod_{l=n-j+1}^n}\beta_l,
\end{equation*}
\begin{multline*}
	\tilde{S}_{n+1} \coloneqq \displaystyle{\sum_{j=0}^\infty}B_{n-j} + 2L\sqrt{\frac{N}{\bar{\beta}}} \displaystyle{\sum_{j=0}^\infty}B_{n-j}\left(\displaystyle{\sum_{k=0}^{j-1}}\sqrt{\eta_{n-k}}\right)
	+ \frac{L^2N}{\bar{\beta}} \displaystyle{\sum_{j=0}^\infty}B_{n-j}\left(\displaystyle{\sum_{k=0}^{j-1}} \sqrt{\eta_{n-k}}\right)^2,
\end{multline*}
\begin{equation}
	S_n \coloneqq \left(\epsilon_a+\sqrt{\tilde{S}_{n+1}}\right)^2,
	\label{def_Sn}
\end{equation}
\begin{equation}
	\eta_o = \min{\left(\frac{2\epsilon_a(1-\lambda)}{L},\frac{1}{\bar{\beta}}\right)} \text{ , } \eta^* = \dfrac{\eta_o}{f_1},
	\label{def_eta_star}
\end{equation}
\begin{equation*}
	\tilde{S} \coloneqq \frac{1}{\bar{\beta}\eta^*} + \dfrac{2L\sqrt{N}(1-\bar{\beta}\eta^*)}{\bar{\beta}^3(\eta^*)^2}
	+ \dfrac{L^2N(1-\bar{\beta}\eta^*)(2-\bar{\beta}\eta^*)}{\bar{\beta}^5 (\eta^*)^3},
\end{equation*}
and 
\begin{equation}
	S \coloneqq \left(\epsilon_a+\sqrt{\tilde{S}}\right)^2.
	\label{def_S}
\end{equation}

We are now ready to state the main theorem of this section:
\begin{theorem}
	Assume \ref{c1}, \ref{local_lip}, \ref{coercive} and that $\R$ is Lojasiewicz. Then $(\theta_n)_{n\in \mathbb{N}}$ produced by Adaptive RMSProp \eqref{RMSProp_eq} converges to $\theta^*$ which is a critical point of $\R$. Let us denote by $\alpha$, $c$ the Lojasiewicz coefficients associated to $\theta^*$. We get the following convergence rates:
	\begin{itemize}
		\item if $\alpha=\frac{1}{2}$, there exists $n_1$ such that for all $n \geq n_1+1$, we get:
		\begin{equation}
			\|\theta_n-\theta^*\| \leq  \dfrac{2\sqrt{S}}{\lambda c \epsilon_a} \left\lvert\R(\theta_{n_1})-\R(\theta^*)\right\rvert \exp{\left( -\frac{\lambda c^2}{2}\displaystyle{\sum_{k=n_1}^{n-1}} \dfrac{\eta_k}{\sqrt{S_{k+1}}} \right)}.
			\label{rms_exp}
		\end{equation}
		If $0<\alpha<\frac{1}{2}$, we get:
		\begin{equation*}
			\|\theta_n-\theta^*\| = \mathcal{O}\left(\left[\displaystyle{\sum_{k=0}^{n-1}} \dfrac{\eta_k}{\sqrt{S_{k+1}}} \right]^{-\frac{\alpha}{1-2\alpha}}\right).
			\label{rms_subexp}
		\end{equation*}
		\item If $\frac{1}{2}<\alpha<1$, then $(\theta_n)_{n\in \mathbb{N}}$ converges in finite time.
	\end{itemize}
	\label{rms_th}
\end{theorem}
\begin{proof}
	See appendix \ref{appendix_rms}.
\end{proof}

\begin{remark}
	The theorem \ref{rms_th} is not a consequence of the works \cite{Absil,Rondepierre} even if it is a descent method. In fact, even without the memory effect ($f_2$), it is not obvious that the angle condition of these papers is satisfied. It is possible to rewrite RMSProp as a preconditionned gradient $B_n d_n = -\nR(\theta_n)$ with:
	\begin{equation*}
		B_n \coloneqq \left(diag\left(\epsilon_a+\sqrt{s_{n+1}^i}\right)\right)_{1\leq i \leq N},
	\end{equation*}
	where $s_{n+1}^i$ denotes the $i$-th component of the vector $s_{n+1}$ and $\left(diag(a_i)\right)_{1\leq i \leq N}$ the diagonal matrix composed of the coefficients $(a_i)_{1\leq i \leq N}$. According to \cite{Absil}, a sufficient condition to check the angle condition consists in showing that the conditioning of $B_n$ is bounded. This conditioning is equal to:
	\begin{equation*}
		\|B_n\| \|B_n^{-1}\| = \dfrac{\epsilon_a+\max_{1 \leq i \leq N}\sqrt{s_{n+1}^i}}{\epsilon_a+\min_{1 \leq i \leq N}\sqrt{s_{n+1}^i}}.
	\end{equation*}
	Therefore $\|B_n\|\|B_n^{-1}\|$ is bounded if the sequence $(s_n)_{n\in \mathbb{N}}$ is also bounded, which is not straightforward.
\end{remark}

The first result in the literature (to the best of our knowledge) to deal with the \textbf{convergence of the iterates of constant learning rate RMSProp \eqref{RMSProp_update} without the classical bounded assumptions} can be deduced as a corollary of the previous theorem. In the next corollary, one can take $f_1=1$ in \eqref{def_eta_star} so that there is no backtracking.

\begin{corollary}
	Assume \ref{c1}, \ref{global_lip}, \ref{coercive} and that $\R$ is Lojasiewicz. If $0<\eta\leq \eta_o$ and $\beta=1-\bar{\beta}\eta$ (for $\bar{\beta}>0$), then $(\theta_n)_{n\in \mathbb{N}}$ produced by RMSProp \eqref{RMSProp_update} converges to $\theta^*$ which is a critical point of $\R$. Let us denote by $\alpha$, $c$ the Lojasiewicz coefficients associated to $\theta^*$.
	We have:
	\begin{itemize}
		\item If $\alpha=\frac{1}{2}$, there exists $n_1$ such that for all $n \geq n_1+1$, we get:
		\begin{equation}
			\|\theta_n-\theta^*\| \leq \dfrac{2\sqrt{S}}{\lambda c \epsilon_a}\exp{\left( -\frac{\lambda c^2\eta_o}{2\sqrt{S}}(n-n_1) \right)}.
			\label{rms_const_exp}
		\end{equation}
		\item If $0<\alpha<\frac{1}{2}$, we have:
		\begin{equation*}
			\|\theta_n-\theta^*\| = \mathcal{O}\left(n^{\frac{-\alpha}{1-2\alpha}}\right).
			\label{rms_const_subexp}
		\end{equation*}
		\item If $\frac{1}{2}<\alpha<1$, then $(\theta_n)_{n\in \mathbb{N}}$ converges in finite time.
	\end{itemize}
	\label{rms_corollary}
\end{corollary}
\begin{proof}
	See appendix \ref{appendix_rms}.
\end{proof}

We will now discuss the dependence of the convergence rates, especially the exponential one \eqref{rms_const_exp} ($\alpha=0.5$), on the different parameters:
\begin{itemize}
	\item The influence of the smallest eigenvalue $c^2$ is the same as GD \cite{Absil} and Momentum \eqref{momentum_exp_cst}.  
	\item For large $L$ ($L \gg 1$), the exponential looks like $\exp{\left(-L^{-7/2}n\right)}$ and the constant in front of it $\frac{2\sqrt{S}}{\lambda c \epsilon_a}$ is of order $L^{5/2}$. So for stiff functions, this estimation seems worse than the one for GD $\exp{\left(-L^{-1}n\right)}$ and Momentum $\exp{\left(-L^{-3}n\right)}$. In flat regions ($L$ small enough in order for $\eta$ to be close to $1/\bar{\beta}$), the rate is independent of $L$. 
	\item For small $\epsilon_a$ ($\epsilon_a \ll 1$), the exponential rate looks like $\exp{\left(-\epsilon_a^{5/2}n\right)}$ and the constant is of order $\epsilon_a^{-5/2}$.  As a result a small value of $\epsilon_a$ (typically $\epsilon_a=10^{-7}$) drastically slows down the convergence. This is not so surprising if one looks at the continuous system. Indeed, when $\epsilon_a$ is close to $0$, the dynamical system \eqref{rms_ode} suffers from a bifurcation: it admits no stationary points when $\epsilon_a=0$. When $\epsilon_a$ is large enough ($\eta$ close to $1/\bar{\beta}$), the exponential rate becomes of the form $\exp{\left(-\epsilon_a^{-1}n\right)}$ and the constant term $\frac{2\sqrt{S}}{\lambda c \epsilon_a}$ does not depend on $\epsilon_a$. Therefore it is maybe a lot more relevant to take big values of $\epsilon_a$ than small ones. These observations go against the usual interpretation for $\epsilon_a$. It is usually seen as a parameter to avoid division by zero \cite{Adam} but our analysis reveals its severe influence on the dynamics. 
	\item For high dimensional problems, the exponential rate is of order $\exp{\left(-N^{-1/2}n\right)}$ and the constant is of the form $\sqrt{N}$.
	This dependence, that does not appear for GD and Momentum, is due to the way we deal with the component-wise operations. It is an open and interesting question to know if it is possible to get rid of the dependence on the dimension.
\end{itemize}

\section{Discussion on the hyperparameters}
\label{hyperparameters}
In the previous sections, we have discussed the role of geometrical characteristics ($L$, $c$ and $N$ imposed by the problem) of the function and optimizers' hyperparameters ($\bar{\beta}$ and $\epsilon_a$ chosen by the users) on the efficiency of the algorithms through its convergence rates. However, the backtracking constants, namely $f_1$, $f_2$ and $\lambda$ (chosen by the users), play a minimal role on the convergence speed. But the global efficiency of the algorithms are not only based on the convergence rates, it also depends on the complexity of the linesearches. In this short section, we give a bound on this complexity.

\paragraph*{}
The expensive part in the backtracking (repeat loops in the algorithms \ref{algo_adaptive_momentum} and \ref{algo_adaptive_rms}) is the evaluation of $\R$. Therefore a relevant measure of this complexity is the number of $\R$ evaluations for each linesearch. As it seems absolutely not trivial to have an idea of this number, the trick is to focus on the mean number of evaluations per linesearch. Let us denote by $C_n$ the number of evaluations until the $n$-th iteration. The next theorem gives a bound on $\frac{C_n}{n}$. 

\begin{theorem}
	The mean number of function evaluations, for Adaptive Momentum or RMSProp, satisfies for $n$ big enough:
	\begin{equation*}
		\dfrac{C_n}{n} \leq \left[1+\dfrac{\log{(f_2)}}{\log{(f_1)}}\right]  
		+\dfrac{1}{n} \dfrac{\log{\left(\frac{f_1^2}{\bar{\beta}\eta^*}\right)}}{\log{(f_1)}},
	\end{equation*}
	where $\eta^*$ is given respectively in \ref{momentum_corollary} and \eqref{def_eta_star}.
	\label{complexity_theorem}
\end{theorem}
\begin{proof}
	See appendix \ref{appendix_complexity}.
\end{proof}

\begin{remark}
	In the previous theorem, $n$ large enough should be understood as the first $n'\geq 0$ such that $f_2^{n'} \frac{1}{\bar{\beta}} \geq \eta^*$. For stiff functions, $\eta^*$ is small, so the rank $n'$ is also small. Then one can consider that the estimation above is valid after few iterations.  
\end{remark}

A priori, the number of evaluations seems to depend strongly on the function itself, especially its stiffness, and the hyperparameters. On the other hand, theorem \ref{complexity_theorem} tells us that the backtracking complexity is around $1+\dfrac{\log{(f_2)}}{\log{(f_1)}}$, with a deviation of $\dfrac{1}{n} \dfrac{\log{\left(\frac{f_1^2}{\bar{\beta}\eta^*}\right)}}{\log{(f_1)}}$. The stiffness $L$ as well as the other hyperparameters only affect the deviation, especially through $\eta^*(L)$ (see remark \ref{momentum_eta_ex} and equation \eqref{def_eta_star}). For Momentum and RMSProp, this deviation only involves logarithm of the hyperparameters and the geometric characteristics of the function. In addition, they are divided by the number of iterations. As a result the deviation is not important even for stiff functions and "bad" hyperparameters. Therefore the backtracking hyperparameters $f_1$ and $f_2$ play the most significant role in the linesearch complexity ($\lambda$ is involved in $\eta^*$ so its role is negligible).

\section{Numerical experiments}
\label{num_experiments}

We finish our analysis with some numerical experiments on a regression and a classification task and compare the results with the classical Momentum and RMSProp. These results strengthen the importance of an adaptive strategy focusing on Lyapunov stability.

\paragraph*{Hyperparameters}
~~\\
In the next numerical results: without further indications, the constant step size algorithms are trained with the Tensorflow default values:
\begin{itemize}
	\item $\eta=10^{-2}$ and $\beta=0.9$ for Momentum (there is no default value for $\beta$ but 0.9 is a common recommendation).
	\item $\eta=10^{-3}$, $\beta=0.999$, $\epsilon_a=10^{-7}$ for RMSProp (default values of Adam \cite{Adam}).
\end{itemize}
For the adaptive optimizers, the value of $\bar{\beta}$ is indicated and we set $\epsilon_a=0.1$ for Adaptive RMSProp since our analysis indicates a better convergence rate for larger $\epsilon_a$. For RMSProp, $\lambda=0.5$ and for Momentum $\lambda=\frac{1}{4\bar{\beta}}<\frac{1}{2\bar{\beta}}$: $\lambda$ is involved both inside the exponential convergence rate and the time step limitation so these choices are "optimal". Finally $f_1=2$ and $f_2=10^4$ in order to compute the biggest learning rate satisfying the dissipation inequality. \\
Adaptive Momentum is used with $\bar{\beta} \in \left\{ 1,2 \right\}$ because in this case, we have explicitly computed the dependence of the convergence rate on $L$. Besides without a-priori knowledge of smoothness, a small value of the friction prevents to restrict the learning rate since $\eta_n \leq \frac{1}{\bar{\beta}}$.

\paragraph*{Stopping criteria}
~~\\
The stopping criteria consists in checking that $\|\nR(\theta_n)\| \leq \epsilon$ with $\epsilon=10^{-4}$ and the maximum number of epochs is set to $n_{max}=200000$. If the stopping criteria is not achieved after $n_{max}$ epochs, the trajectory is considered non convergent. Many trainings with different initializations are performed to be statistically significant.

\paragraph*{Sonar classification}
~~\\
The first example is the Sonar classification dataset consisting in identifying between a rock and a metal cylinder from sonar signals. There are 104/104 training/testing points: the original database is available in the following \href{https://archive.ics.uci.edu/dataset/151/connectionist+bench+sonar+mines+vs+rocks}{repository} \footnote{https://archive.ics.uci.edu/dataset/151/connectionist+bench+sonar+mines+vs+rocks}. The original database is made up of 208 data points but there is no standard splitting between training and testing dataset. We use the function train\_test\_split of scikit-
learn with test\_size= 0.5 and random\_state=7 to ensure a balance between the labels. We apply a standard normalization for the entries (zero mean and deviation of one) and a label encoding for the output. All the pretreatments are fully described in a jupyter lab and the files obtained are listed in a git repository. \\
This example is trained on a two layer feed-forward neural network with gelu activation function on the hidden layer (30 neurons) and a sigmoid on the output layer. So $N=1861$. We use the following approximation of the gelu fonction \cite{gelu}:
\begin{equation*}
	gelu(x) \approx x\sigma(1.702x),
\end{equation*}
where $\sigma$ denotes the sigmoid function. The 10000 initial points are sampled from the Xavier initializer. Let us briefly recall the Xavier initialization. The bias are zero. The weights of the $l$-layer are given by:
\begin{equation}
	W_l \sim \dfrac{1}{\sqrt{n_{in}}}\mathcal{N}(0,1),
	\label{Xavier}
\end{equation}
where $W_l$ is the weight matrix, $n_{in}$ is the input dimension of the $l$-layer and $\mathcal{N}(0,1)$ is the standard gaussian distribution.\\

The table \ref{Sonar_ICML_tab} displays the median on 10000 trainings of the following indicators: percent of non convergent trajectories, the training and testing accuracy and the execution time.\\
Momentum and Adaptive Momentum gives the same accuracies but Adaptive Momentum is much faster by at least a factor of 10 (1-15s vs 174s). The default RMSProp may be faster than Adaptive RMSProp depending of the value of $\bar{\beta}$. However, if one takes another value of the learning rate ($\eta=0.1$ that ensures convergence), the performances completely collapses: 52-54 \%  of recognition for RMSProp with constant learning rate. This is typical of a Lyapunov instability that may move the current weights away from a good minimum \cite{Bilel}. Therefore \textbf{speeding-up the training is not always relevant if stability is not ensured (even if the algorithm converges)}.

\begin{table}[h!]
	\centering
	\tbl{Median performances for Sonar classification ($\epsilon=10^{-4}$) on 10000 points: distribution of non convergent trajectories, percent of well classified points and training time are displayed.}
	{\begin{tabular}{lcccc}
			\toprule
			\begin{bf} Algos \end{bf} & \begin{bf}Non-cv\end{bf} & \begin{bf}Training\end{bf} &  \begin{bf}Testing\end{bf}  & \begin{bf}Time(s)\end{bf}\\ \midrule
			Momentum  & 0 & 100 & 89 & 174 \\ \midrule
			RMSProp  & 0 & 100 & 89 & 3 \\ \midrule
			RMSProp ($\eta=0.1$)  & 0 & 52 & 54 & 0   \\ \midrule
			Adaptive Momentum ($\bar{\beta}=1$) & 0 & 100 & 89 & 14 \\ \midrule
			Adaptive Momentum ($\bar{\beta}=2$) & 0 & 100 & 89 & 57 \\ \midrule
			Adaptive RMSProp ($\bar{\beta}=1$) & 0 & 100 & 89 & 1 \\ \midrule
			Adaptive RMSProp ($\bar{\beta}=10$) & 0 & 100 & 89 & 15 \\ \bottomrule
			
	\end{tabular}}
	\label{Sonar_ICML_tab}
\end{table}

\paragraph*{Boston regression}
~~\\
The second benchmark is the Boston Housing problem (\href{https://www.kaggle.com/datasets/altavish/boston-housing-dataset/data}{database}\footnote{https://www.kaggle.com/datasets/altavish/boston-housing-dataset/data}) consisting in predicting the median price of houses in Boston from 13 features (average number of rooms per dwelling, distances to five Boston employment centres, ...). We apply a standard normalization for the entries and a min-max normalization for the output. \\
The benchmark is trained on a three layer feed-forward neural network (linear output layer). There are 15 neurons with tanh activations for each hidden layer. So $N=466$. The 300 initial points are sampled from the Xavier initializer.\\

The table \ref{Boston_ICML_tab} displays the median on 300 trainings of the following indicators: percent of non convergent trajectories, the training and testing $L^2$ norm and the execution time.\\
Default Momentum does not converge 30\% of the time and even when it converges it is much slower than Adaptive Momentum for the two values of the friction coefficient analyzed. Default RMSProp does not converge almost everytime (98\%) but when it converges it gives a very good training error. If one increases the learning rate ($\eta=0.1$) all the trajectories converge faster than Adaptive RMSProp but the training and testing error are 100 times bigger: this is another illustration of \textbf{instability}. Its adaptive version gives good errors for the two settings: let us recall that the stability is verified for any value of the friction coefficient.     

\begin{table}[h!]
	\centering
	\tbl{Median performances for Boston regression ($\epsilon=10^{-4}$) on 300 points: distribution of non convergent trajectories, performances and training time are displayed}
	{\begin{tabular}{lcccc}
		\toprule
		\begin{bf} Algos \end{bf} & \begin{bf}Non-cv\end{bf} & \begin{bf}Training\end{bf} &  \begin{bf}Testing\end{bf}  & \begin{bf}Time(s)\end{bf}\\ \midrule
		Momentum  & 29.7 & 3e-4 & 5e-3 & 2216 \\ \midrule
		RMSProp & 98 & 4.5e-6 & 1e-2  & 1634 \\ \midrule
		RMSProp ($\eta=0.1$) & 0 & 1e-2 & 1e-2  & 183 \\ \midrule
		Adaptive Momentum ($\bar{\beta}=1$) & 0 & 2.8e-4 & 5e-3 & 410 \\ \midrule
		Adaptive Momentum ($\bar{\beta}=2$) & 0 & 2.8e-4 & 5e-3 & 876 \\ \midrule
		Adaotive RMSProp ($\bar{\beta}=1$) & 0 & 2.8e-4 & 5e-3 & 1100 \\ \midrule
		Adaptive RMSProp ($\bar{\beta}=10$) & 0 & 2.8e-4 & 5e-3 & 769 \\ \bottomrule
	\end{tabular}}
	\label{Boston_ICML_tab}
\end{table}

\section{Conclusion}
\label{conclusion}
In this paper, we suggest an adaptive time step strategy for two popular optimizers, namely Momentum and RMSProp. These new algorithms appear as a generalization of the Armijo rule for gradient descent by reinterpreting it as the preservation of a continuous dissipation inequality. These adaptive optimizers come with convergence results on the iterates and do not need the global lipshitz assumption that is very restrictive for neural networks, see \cite{compute_bound_L}. As corollary of this analysis, two new theoretical results on constant step size Momentum and RMSProp are derived under weak assumptions. Numerical experiments confirm that these new adaptive time step strategies converge for different choices of the friction parameter and achieve good performances thanks to the Lyapunov stability. 
	
\section*{Data availability statement}
The different implementations are available on the following repositories:
\begin{itemize}
	\item the C++ code for the network training is available at \href{https://github.com/bbensaid30/COptimizers.git}{this address}\footnote{https://github.com/bbensaid30/COptimizers.git}.
	\item The data and their pretreatments can be found on  \href{https://github.com/bbensaid30/ML_data.git}{this repository}\footnote{https://github.com/bbensaid30/ML\_data.git}.
\end{itemize}	
				
\section*{Funding}
This work was supported by CEA/CESTA and LRC Anabase.

\section*{Disclosure statement}
The authors report there are no competing interests to declare.
				
\bibliographystyle{tfs}
\bibliography{bibliography}

\newpage
				
\appendix

\section{Algorithms}
\label{appendix_algorithms}

Here we write the adaptive time step algorithms presented in section \ref{generalize_armijo} in a form that is directly implementable. 

\begin{algorithm}[ht]
	\caption{Adaptive Momentum}
	\begin{algorithmic}
		\REQUIRE initial values $\theta_0$, $\bar{\beta}\geq 1$, $f_1>1$, $f_2>1$, $\lambda \in ]0,1/2\bar{\beta}[$.
		
		\STATE $\theta \leftarrow \theta_0$
		\STATE $v \leftarrow 0$
		\STATE $\eta \leftarrow 1/\bar{\beta}$
		\FOR{$n \geq 0$}
		\STATE $V_0 \leftarrow \R(\theta)+\frac{\|v\|^2}{2}$
		\STATE $\theta_0 \leftarrow \theta$
		\STATE $v_0 \leftarrow v$
		\REPEAT
		\STATE $v \leftarrow (1-\bar{\beta}\eta)v + \eta \nR(\theta)$
		\STATE $\theta \leftarrow \theta - \eta v$
		\STATE $V \leftarrow \R(\theta)+\frac{\|v\|^2}{2}$
		\IF{$V-V_0 > -\lambda \bar{\beta} \eta \|v\|^2 $}
		\STATE $\eta \leftarrow \eta/f_1$
		\STATE $\theta \leftarrow \theta_0$
		\STATE $v \leftarrow v_0$
		\ENDIF
		\UNTIL{$V-V_0 \leq -\lambda \bar{\beta} \eta \|v\|^2$}
		\STATE $\eta \leftarrow \min{\left(f_2 \eta, 1/\bar{\beta}\right)}$
		\STATE $n \leftarrow n+1$
		\ENDFOR
	\end{algorithmic}
	\label{algo_adaptive_momentum}
\end{algorithm}

\begin{algorithm}[ht]
	\caption{Adaptive RMSProp}
	\begin{algorithmic}
		\REQUIRE initial values $\theta_0$, $\bar{\beta}>0$, $\epsilon_a>0$, $f_1>1$, $f_2>1$, $\lambda \in ]0,1[$.
		
		\STATE $\theta \leftarrow \theta_0$
		\STATE $s \leftarrow 0$
		\STATE $\eta \leftarrow 1/\bar{\beta}$
		\FOR{$n \geq 0$}
		\STATE $\R_0 \leftarrow \R(\theta)$
		\STATE $\theta_0 \leftarrow \theta$
		\STATE $s_0 \leftarrow s$
		\REPEAT
		\STATE $s \leftarrow (1-\bar{\beta}\eta)s + \bar{\beta}\eta \nR(\theta)^2$
		\STATE $\theta \leftarrow \theta - \eta \dfrac{\nR(\theta)}{\epsilon_a+\sqrt{s}}$
		\STATE $\R \leftarrow \R(\theta)$
		\IF{$\R-\R_0 > - \lambda \eta \left|\left|\dfrac{\nR(\theta_0)}{\sqrt{\epsilon_a+\sqrt{s}}}\right|\right|^2 $}
		\STATE $\eta \leftarrow \eta/f_1$
		\STATE $\theta \leftarrow \theta_0$
		\STATE $s \leftarrow s_0$
		\ENDIF
		\UNTIL{$\R-\R_0 \leq -\lambda \eta \left|\left|\dfrac{\nR(\theta)}{\sqrt{\epsilon_a+\sqrt{s}}}\right|\right|^2$}
		\STATE $\eta \leftarrow \min{\left(f_2 \eta, 1/\bar{\beta}\right)}$
		\STATE $n \leftarrow n+1$
		\ENDFOR
	\end{algorithmic}
	\label{algo_adaptive_rms}
\end{algorithm}

\begin{remark}
	To obtain a Lyapunov inequality for Momentum, we have just to choose an inequality that is consistent with \eqref{dissipation_momentum}. For the left hand side $\dfrac{dV(\theta(t),v(t))}{dt}$, we choose an Euler discretization $\dfrac{V(\theta_{n+1},v_{n+1})-V(\theta_n,v_n)}{\eta}$. For the dissipation rate $-\lambda \bar{\beta} \|v(t)\|^2$, two natural choices are possible, namely $-\lambda \bar{\beta} \|v_n\|^2$ and $-\lambda \bar{\beta} \|v_{n+1}\|^2$. We choose the second one in this paper since $v_{n+1}$ is already involved in the update of the parameters $\theta_n$, see \eqref{Momentum_update}. It would be an interesting question to know if the choice $-\lambda \bar{\beta} \|v_n\|^2$ improves or not the convergence rates of theorem \ref{momentum_th}. The same remark can be done for RMSProp about $s_{n+1}$ in the dissipation term. But these questions are beyond the scope of this paper.\\
	Dealing with consistent issues, we can highlight a third drawback for the Armijo generalization suggested in \cite{armijo_momentum}. In fact, the update for Momentum in this paper is not even consistent with the ODE \eqref{Momentum_ode}.  
\end{remark}

\section{Momentum proofs}
\label{appendix_momentum}

In this section, we prove lemma \ref{lemma_delta_momentum} and a two step Gronwall inequality- \ref{Gronwall_two_step}. Then we present the proofs of theorem \ref{momentum_th} and corollary \ref{momentum_corollary}. Let us denote by $\eta^* = \frac{\eta_o}{f_1}$, where $\eta_o$ is given by lemma \ref{Momentum_eta_inf}.

\begin{proof}[Proof of lemma \ref{lemma_delta_momentum}]
	Let us recall that $\Delta_n \coloneqq \|\theta_{n+1}-\theta_n\|$ and $z_n\coloneqq (\theta_n,v_n)$. If $\Delta_{n+1}=0$, the inequality is obvious so we can assume that $\Delta_{n+1}>0$. Without loss of generality, we can assume that $V(z^*)=0$.
	First we will compare $\nabla V(z_n)$ with $v_n$. As $\nabla V(\theta,v) = (\nR(\theta), v)^T$, we can write:
	\begin{equation}
		\|\nabla V(z_n)\| \leq \|\nR(\theta_n)\| + \|v_n\| = \left|\left|\dfrac{v_{n+1}-\beta_n v_n}{\eta_n}\right|\right| + \|v_n\| = \frac{1}{\eta_n}\|v_{n+1}\|+\left(1+\frac{\beta_n}{\eta_n}\right)\|v_n\|.
		\label{gradientV_compare}
	\end{equation}
	Now we will compare $\Delta_n$ with $V(z_n)$:
	\begin{equation*}
		V(z_n)^{\alpha}-V(z_{n+1})^{\alpha} = \int_{V(z_{n+1})}^{V(z_n)}\alpha \omega^{\alpha-1}d\omega \geq \alpha \int_{V(z_{n+1})}^{V(z_n)}V(z_n)^{\alpha-1}d\omega
	\end{equation*}
	since $V(z_{n+1}) \leq \omega \leq V(z_n)$.
	Consequently:
	\begin{equation}
		V(z_n)^{\alpha}-V(z_{n+1})^{\alpha} \geq \alpha  V(z_n)^{\alpha-1}\left[ V(z_n)-V(z_{n+1})\right].
		\label{diff_V_alpha}
	\end{equation}
	For $z\in \mathcal{U}_{z^*}$ and $\nabla V(z) \neq 0$, we can apply Lojasiewicz inequality:
	\begin{equation}
		V(z)^{\alpha-1} \geq \dfrac{c}{\|\nabla V(z)\|}.
		\label{V_alpha}
	\end{equation}
	As $z_n \in \mathcal{U}_{z^*}$ and $\nabla V(z_n)\neq 0$ (if it was not the case, $\nR(\theta_n)=0=v_n$ so $\Delta_{n+1}=0$), we combine \eqref{diff_V_alpha} and \eqref{V_alpha} to obtain:
	\begin{equation*}
		V(z_n)^{\alpha}-V(z_{n+1})^{\alpha} \geq \dfrac{\alpha c}{\|\nabla V(z_n)\|}\left[V(z_n)-V(z_{n+1})\right].
	\end{equation*}
	Using the bound \eqref{gradientV_compare} in the previous inequality leads to:
	\begin{equation*}
		V(z_n)^{\alpha}-V(z_{n+1})^{\alpha} \geq \dfrac{\alpha c\left[V(z_n)-V(z_{n+1})\right]}{\frac{1}{\eta_n}\|v_{n+1}\|+\left(1+\frac{\beta_n}{\eta_n}\right)\|v_n\|}.
	\end{equation*}
	Using the dissipation condition \eqref{Momentum_condition} and $\|v_{n+1}\|=\dfrac{\Delta_{n+1}}{\eta_n}$ (see \eqref{Momentum_eq}) in the previous inequality, we get:
	\begin{multline*}
		V(z_n)^{\alpha}-V(z_{n+1})^{\alpha} \geq \dfrac{\lambda \alpha c \bar{\beta}\eta_n \|v_{n+1}\|^2}{\frac{1}{\eta_n}\|v_{n+1}\|+\left(1+\frac{\beta_n}{\eta_n}\right)\|v_n\|} \geq \dfrac{\lambda \alpha c \bar{\beta}\eta_{n-1}\eta_n\Delta_{n+1}^2}{\eta_{n-1}\Delta_{n+1}+\eta_n(\eta_n+\beta_n)\Delta_n} \\
		\geq \dfrac{\lambda \alpha c \bar{\beta}\eta_{n-1}\eta_n\Delta_{n+1}^2}{A_n(\Delta_n+\Delta_{n+1})}.
	\end{multline*}
	where $A_n$ is given in lemma \ref{lemma_delta_momentum}. We can deduce the following inequality on $\Delta_{n+1}$:
	\begin{equation*}
		\Delta_{n+1} \leq \sqrt{\frac{1}{2}(\Delta_n+\Delta_{n+1})\times \dfrac{2A_n}{\lambda \alpha c \bar{\beta}\eta_{n-1}\eta_n}\left[V(z_n)^{\alpha}-V(z_{n+1})^{\alpha}\right]}.
	\end{equation*}
	Using the fact that $\sqrt{uv} \leq \dfrac{u+v}{2}$ for $u,v\geq 0$, we get the following inequality:
	\begin{equation*}
		\Delta_{n+1} \leq \frac{1}{4}(\Delta_n+\Delta_{n+1}) + \dfrac{A_n}{\lambda \alpha c \bar{\beta}\eta_{n-1}\eta_n}\left[V(z_n)^{\alpha}-V(z_{n+1})^{\alpha}\right],
	\end{equation*}
	By multiplying by $2$ the previous equation and subtracting $\Delta_{n+1}$ on each side, we obtain what we want.
\end{proof}

\begin{lemma}
	Let $(u_n)_{n\in \mathbb{N}}$ be a lower bounded decreasing sequence and $(w_n)_{n\in \mathbb{N}}$ a positive one such that:
	\begin{equation*}
		\forall n \geq 2, u_{n-2}-u_n \geq w_{n-1}u_{n-1}^{\gamma},
	\end{equation*}
	for some $\gamma \geq 1$.\\
	If $\gamma=1$ then for $n\geq 2$, we have:
	\begin{equation}
		u_n \leq e^{\sqrt{u_0 u_1}} \exp{\left(-\frac{1}{2}\displaystyle{\sum_{k=1}^{n-1}w_k}\right)}.
		\label{gronwall2_exp}
	\end{equation}
	If $\gamma>1$, then for $n\geq 2$, we get:
	\begin{equation*}
		u_n \leq \left[\dfrac{2}{u_0^{1-\gamma}+u_1^{1-\gamma}+(\gamma-1)\displaystyle{\sum_{k=1}^{n-1}w_k}}\right]^{\frac{1}{\gamma-1}}.
	\end{equation*}
	If $0<\gamma<1$ and $\displaystyle{\sum_{k\geq 0}} w_k$ diverges, then $(u_n)_{n\in \mathbb{N}}$ is stationary.  
	\label{Gronwall_two_step}
\end{lemma}

\begin{proof}
	Let us discuss the case $\gamma=1$. By the mean value theorem, it exists a real $c \in [u_n, u_{n-2}]$ that satisfies:
	\begin{equation*}
		\ln{(u_n)}-\ln{(u_{n-2})} = \int_{u_{n-2}}^{u_n}\dfrac{dx}{x} = \dfrac{u_n-u_{n-2}}{c}.
	\end{equation*}
	The previous inequality combined with $u_n-u_{n-2} \leq -w_{n-1}u_{n-1}^{\gamma}$ gives:
	\begin{equation*}
		\ln{(u_n)}-\ln{(u_{n-2})} \leq -\dfrac{w_{n-1}u_{n-1}}{u_n}.
	\end{equation*}
	As $(u_n)_{n\in \mathbb{N}}$ is decreasing and $(w_n)_{n\in \mathbb{N}}$ positive:
	\begin{equation*}
		\ln{(u_n)}-\ln{(u_{n-2})} \leq -\dfrac{w_{n-1}u_{n-1}}{u_n} \leq -v_{n-1}.
	\end{equation*}
	By summing the inequality from $k=0$ to $n-1$, we get:
	\begin{equation*}
		\displaystyle{\sum_{k=0}^{n-1}}\ln{(u_{k+2})}-\ln{(u_k)} \leq -\displaystyle{\sum_{k=0}^{n-1}}w_{k+1}.
	\end{equation*}
	The telescopic sum leads to:
	\begin{equation*}
		\ln{(u_{n+1})}+\ln{(u_n)} \leq \ln{(u_0)}+\ln{(u_1)}-\displaystyle{\sum_{k=0}^{n-1}}w_{k+1}.
	\end{equation*}
	As $(u_n)_{n\in \mathbb{N}}$ is decreasing, we can deduce that:
	\begin{equation*}
		2\ln{(u_{n+1})}\leq \ln{(u_0)}+\ln{(u_1)}-\displaystyle{\sum_{k=0}^{n-1}}w_{k+1},
	\end{equation*}
	which is equivalent to the expected inequality \eqref{gronwall2_exp}. The case $\gamma>1$ is essentially the same by replacing the first equality by:
	\begin{equation*}
		\dfrac{u_n^{1-\gamma}-u_{n-2}^{1-\gamma}}{1-\gamma} = \int_{u_{n-2}}^{u_n}\dfrac{dx}{x^{\gamma}} = \dfrac{u_n-u_{n-2}}{c^{\gamma}}.
	\end{equation*}
	Finally for $0<\gamma<1$, the same steps lead to:
	\begin{equation*}
		2u_{n+1}^{1-\gamma} \leq u_0^{1-\gamma}+u_1^{1-\gamma}-(1-\gamma) \displaystyle{\sum_{k=0}^{n-1}}w_{k+1}.
	\end{equation*}
	Let us consider a integer $n'\geq 1$ that satisfies:
	\begin{equation*}
	\displaystyle{\sum_{k=0}^{n'-1}}w_{k+1} + \frac{2}{1-\gamma}l^{1-\gamma} \geq \dfrac{u_0^{1-\gamma}+u_1^{1-\gamma}}{1-\gamma} \Leftrightarrow 
	u_0^{1-\gamma}+u_1^{1-\gamma}-(1-\gamma) \displaystyle{\sum_{k=0}^{n-1}}w_{k+1} \leq 2l^{1-\gamma}.
	\end{equation*}
	where $l=\inf u_n$. Such an integer exists since the sum of $(w_n)_{n\in \mathbb{N}}$ diverges. So for $n\geq n'$, $u_{n+1}\leq l$, which implies $u_{n+1}=l$. Then $(u_n)_{n\in \mathbb{N}}$ is stationary. 
\end{proof}

\begin{proof}[Proof of theorem \ref{momentum_th}]
	This proof is organized in three parts. First, we will prove that the sequence $(v_n)_{n\in \mathbb{N}}$ converges to $0$ and extract a subsequence $(\theta_{\phi(n)})_{n\in \mathbb{N}}$ that admits a critical point $\theta^*$ as its limit using lemma \ref{Momentum_eta_inf}. In the second part, the convergence of $(\theta_n)_{n\in \mathbb{N}}$ is established. In fact, the previous subsequence enters in an arbitrary small neighborhood of $\theta^*$ (where the Lojasiewicz inequality is satisfied) and thanks to lemma \ref{lemma_delta_momentum}, we show that all the iterates $(\theta_n)_{n\in \mathbb{N}}$ are stuck in this region. Finally, we derive a Gronwall type inequality for $(V(z_n))_{n\in \mathbb{N}}$ to apply lemma \ref{Gronwall_two_step}. Using a bound established in part 2 between $\|\theta_n-\theta^*\|$ and $V(z_n)$, we deduce the convergence rates of theorem \ref{momentum_th}.
	
	\begin{proofpart}
		In this part, let us show that $(z_n)_{n\in \mathbb{N}} \coloneqq \left((\theta_n, v_n)\right)_{n\in \mathbb{N}}$ is \textbf{bounded and admits a convergence subsequence with a limit of the form $(\theta^*,0)$ where $\theta^*$ is a critical point of $\R$}. \\
		If $(\theta_n, v_n)_{n\in \mathbb{N}}$ is unbounded, then $V(\theta_n, v_n) \to +\infty$ because $V$ is coercive, which contradicts the fact that $\left(V(\theta_n, v_n)\right)_{n\in \mathbb{N}}$ is decreasing \eqref{Momentum_condition}. So $(\theta_n, v_n)_{n\in \mathbb{N}}$ is bounded. \\
		Now, let us show that $v_n \to 0$. As $\left(V(z_n)\right)_{n\in \mathbb{N}}$ is decreasing and lower bounded, then $\left(V(z_n)\right)_{n\in \mathbb{N}}$ converges and we deduce from the dissipation condition \eqref{Momentum_condition} that $\eta_n \|v_{n+1}\|^2 \to 0$. We want to conclude that $v_n\to 0$. In this perspective, we will prove that $\inf \eta_n >0$. \\
		To do this, let us take a look at the time step update. As $\R$ is locally Lipshitz continuous, the lemma \ref{Momentum_eta_inf} on the compact $K=\overline{\{(\theta_n,v_n), n\geq 0\}}$ gives a $\eta_o>0$ such that:
		\begin{equation*}
			\eta_n \leq \eta_o \implies V(z_{n+1})-V(z_n) \leq -\lambda \bar{\beta} \eta_n \|v_{n+1}\|^2.
		\end{equation*}
		The algorithm starts the first iteration with a time step $\eta_{init}=1/\bar{\beta}$, and at the iteration $n\geq 1$ we begin with a time step $f_2 \eta_{n-1}$. We have two complementary cases that may occur:
		\begin{enumerate}
			\item We begin with a time step $f_2 \eta_{n-1}$ smaller than $\eta_o$. So the dissipation inequality \eqref{Momentum_condition} is already satisfied and supplementary computations are not needed to escape the backtracking loop. Therefore $\eta_n=f_2 \eta_{n-1}$.
			\item If  $f_2 \eta_{n-1} > \eta_o$, $f_2 \eta_{n-1}$ may be reduced by a factor $f_1$ several times. In the worst case, the algorithm has not found any solution greater than $\eta_o$ and we have to divide it one more time by $f_1$ so that $\eta_n \geq \frac{\eta_o}{f_1}=\eta^*$.
		\end{enumerate}
		As a result, the loop finishes with a time step $\eta_n \geq \min(f_2 \eta_{n-1},\eta^*)$ if $n > 0$ and $\eta_0 \geq \min(\eta_{init},\eta^*)$ if $n=0$. 
		By induction we have for $n\geq 0$:
		\begin{equation*}
			\eta_n \geq \min\left(f_2^n \eta_{init},\eta^*\right).
		\end{equation*}
		As $f_2>1$ there exists $n_0 \geq 0$ such that $ \forall n\geq n_0, f_2^n \eta_{init} \geq \eta^*$. Therefore:
		\begin{equation*}
			\displaystyle\inf_{n\geq n_0}{\eta_n} = \eta^* > 0.
		\end{equation*}
		Combining $\eta_n \|v_{n+1}\| \to 0$ with $\inf \eta_n >0$, we get $v_n \to 0$. \\
		As $(\theta_n, v_n)_{n\in \mathbb{N}}$ is bounded, the Bolzano-Weirstrass theorem gives the existence of a subsequence $\left(\theta_{\phi(n)}, v_{\phi(n)}\right)_{n\in \mathbb{N}}$ that converges to $z^* \coloneqq (\theta^*, v^*)$. First $v^*=0$ since $v_n \to 0$. We will justify that $\theta^*$ is a critical point of $\nR$.
		By the first update rule of Momentum \eqref{Momentum_eq} and the triangle inequality, we get (have in mind that $v_n \to 0$) for $n \geq n_0$:
		\begin{equation*}
			\|\nR(\theta_{\phi(n)})\| \leq \dfrac{\|v_{\phi(n)+1}\|+\beta_{\phi(n)} \|v_{\phi(n)}\|}{\eta_{\phi(n)}} \leq \dfrac{\|v_{\phi(n)+1}\|+(1-\bar{\beta}\eta^*) \|v_{\phi(n)}\|}{\eta_*} \to 0.
		\end{equation*}
		since $\eta_n \geq \eta^*$. By continuity of $\nR$ (consequence of assumption \ref{local_lip}), $\nR(\theta^*)=0$. Without loss of generality, we can assume that $V(z^*)=0$ in the rest of the proof.
	\end{proofpart}
	
	\begin{proofpart}
		In this part, we \textbf{show that the sequence $(\theta_n)$ converges. To do this, we need to focus on the sum $\displaystyle{\sum_{n\geq n_0}} \Delta_{n+1}$} where $\Delta_n \coloneqq \|\theta_{n+1}-\theta_n\|$. \\ 
		From now on, let $n\geq n_0$ and denote by $\mathcal{U}$ the neighborhood of $z^*$ where the Lojasiewicz inequality is satisfied. According to lemma \ref{lemma_delta_momentum}, we can deduce for $z_n \in \mathcal{U}$:
		\begin{equation*}
			\Delta_{n+1} \leq \frac{1}{2}(\Delta_n-\Delta_{n+1})+\dfrac{2A_n}{\lambda\alpha c \bar{\beta} \eta_{n-1} \eta_n}\left[V(z_n)^{\alpha}-V(z_{n+1})^{\alpha}\right].
		\end{equation*}
		First notice that $\dfrac{A_n}{\eta_{n-1}\eta_n}$ is uniformly bounded by $\frac{1}{\eta^*}+(1+\bar{\beta})f_2$:
		\begin{equation*}
			\dfrac{A_n}{\eta_{n-1}\eta_n} \leq 
			\left\{
			\begin{array}{ll}
				\frac{1}{\eta_n} \leq \frac{1}{\eta^*}  \text{ if } A_n=\eta_n, \\
				\dfrac{\eta_n+\beta_n}{\eta_{n-1}} \leq \dfrac{\eta_n+1+\bar{\beta}\eta_n}{\eta_{n-1}} \leq \frac{1}{\eta_{n-1}} + (1+\bar{\beta})\frac{\eta_n}{\eta_{n-1}}
				\text{ else. }
			\end{array}
			\right.
		\end{equation*}
		It remains to note that $\frac{\eta_n}{\eta_{n-1}} \leq f_2$ because we begin the research of the new time step with the learning rate $f_2 \eta_{n-1}$. For simplification, let us denote $A^* \coloneqq \dfrac{2}{\lambda \alpha c \bar{\beta}}\left[\frac{1}{\eta^*}+(1+\bar{\beta})f_2\right]$.\\\\
		If $z_p, \dots , z_{q-1} \in \mathcal{U}$, we can sum the following inequalities from $k=p$ to $q-1$,
		\begin{equation*}
			\Delta_{k+1} \leq \frac{1}{2}(\Delta_k-\Delta_{k+1})+A^*\left[V(z_k)^{\alpha}-V(z_{k+1})^{\alpha}\right],
		\end{equation*}
		and get:
		\begin{equation}
			\displaystyle{\sum_{k=p}^{q-1}} \Delta_{k+1} \leq \dfrac{\Delta_p-\Delta_{q}}{2} + A^* \left[V(z_p)^{\alpha}-V(z_q)^{\alpha}\right].
			\label{momentum_sum_delta}
		\end{equation}
		
		Let $r>0$ be such that $B_r(\theta^*) \subset \mathcal{U}$. Given that $\theta^*$ is a accumulation point of $(\theta_n)_{n\in \mathbb{N}}$ and $(V(z_n))_{n\in \mathbb{N}}$ converges, it exists $n_1 \geq n_0$ such that:
		\begin{equation}
			\begin{array}{lll}
				\|\theta_{n_1}-\theta^*\|<\frac{r}{3}, \\\\
				\forall q \geq n_1: A^*\left[V(z_{n_1})^{\alpha}-V(z_{q})^{\alpha}\right] < \frac{r}{3}, \\\\
				\frac{\Delta_{n_1}}{2} < \frac{r}{3}.
				\label{r_inf}
			\end{array}
		\end{equation}
		
		Let us show that $\theta_n \in B_r(\theta^*)$ for all $n\geq n_1$. By contradiction assume that it is not the case: $\exists n\geq n_1, \theta_n \notin B_r(\theta^*)$. The set $\left\{ n\geq n_1, \theta_n \notin B_r(\theta^*)\right\}$ is a non empty bounded by below part of $\mathbb{N}$. So we can consider the minimum of this set that we denote by $q$. As a result, $\forall n_1 \leq n < q$, $\theta_n \in \mathcal{U}$ so we can apply \eqref{momentum_sum_delta} and \eqref{r_inf}:
		\begin{equation*}
			\displaystyle{\sum_{n=n_1}^{q-1}}\Delta_{n+1} \leq \dfrac{\Delta_{n_1}}{2} + A^* \left[V(z_{n_1})^{\alpha}-V(z_q)^{\alpha}\right] \leq \frac{2r}{3}. 
		\end{equation*}
		This implies:
		\begin{equation*}
			\|\theta_q-\theta^*\| \leq \displaystyle{\sum_{n=n_1}^{q-1}}\Delta_{n+1} + \|\theta_{n_1}-\theta^*\| < r.
		\end{equation*}
		This is a contradiction because $\|\theta_q-\theta^*\| \geq r$. As $r$ is arbitrary small this shows the convergence of $(\theta_n)$.
	\end{proofpart}
	
	\begin{proofpart}
		Now, it remains to \textbf{obtain the convergences rates}.\\ 
		From now on, let $n\geq n_1$. As $\theta_n \in \mathcal{U}$, by applying \eqref{momentum_sum_delta}, we get:
		\begin{equation}
			\|\theta_n-\theta^*\| = \left|\left|\displaystyle{\sum_{k=n}^{+\infty}} (\theta_k-\theta_{k+1})\right|\right| \leq \displaystyle{\sum_{k=n}^{+\infty}} \Delta_{k+1} \leq \dfrac{\Delta_n}{2}+A^*V(z_n)^{\alpha}.
			\label{theta_V}
		\end{equation}
		
		To obtain the convergence speed, we have to derive a Gronwall inequality on $V(z_n)$.
		Let us remember from the proof of lemma \ref{lemma_delta_momentum} that:
		\begin{equation*}
			\|\nabla V(z_n)\| \leq \frac{\Delta_{n+1}}{\eta_n^2}+\frac{1}{\eta_{n-1}}\left(1+\frac{\beta_n}{\eta_n}\right)\Delta_n.
		\end{equation*}
		The Lyapunov inequality \eqref{Momentum_condition} with $\|v_{n+1}\| = \frac{\Delta_{n+1}}{\eta_n}$ trigger the following inequality:
		\begin{equation*}
			V(z_{n+1})-V(z_n) \leq -\lambda \bar{\beta}\eta_n \dfrac{\Delta_{n+1}^2}{\eta_n^2} = -\lambda \bar{\beta} \dfrac{\Delta_{n+1}^2}{\eta_n},
		\end{equation*}
		that we can use in the previous inequality to get:
		\begin{multline*}
			\|\nabla V(z_n)\|^2 \leq \left[\frac{\Delta_{n+1}}{\eta_n^2}+\frac{1}{\eta_{n-1}}\left(1+\frac{\beta_n}{\eta_n}\right)\Delta_n\right]^2 \leq \frac{2\Delta_{n+1}^2}{\eta_n^4}+\frac{2}{\eta_{n-1}^2}\left(1+\frac{\beta_n}{\eta_n}\right)^2\Delta_n^2, \\ 
			\leq \frac{2}{\lambda \bar{\beta}\eta_n^3}\left[V(z_n)-V(z_{n+1})\right]+\frac{2}{\lambda \bar{\beta}\eta_{n-1}}\left(1+\frac{\beta_n}{\eta_n}\right)^2\left[V(z_{n-1})-V(z_n)\right] \\
			\leq c^2B_n\left[V(z_{n-1})-V(z_{n+1})\right].
		\end{multline*}
		where
		\begin{equation*}
			B_n \coloneqq \frac{2}{\lambda \bar{\beta} c^2} \max{\left(\frac{1}{\eta_n^3}, \frac{1}{\eta_{n-1}} \left(1+\frac{\beta_n}{\eta_n}\right)^2\right)}.
		\end{equation*}
		The previous expression of $B_n$ can be written in the form \eqref{expr_Bn} of theorem \ref{momentum_th} by replacing $\beta_n$ by $1-\bar{\beta}\eta_n$. 
		As $z_n \in \mathcal{U}$, the Lojasiewicz inequality gives:
		\begin{equation*}
			c^2 V(z_n)^{2(1-\alpha)} \leq \|\nabla V(z_n)\|^2,
		\end{equation*}
		which leads to:
		\begin{equation*}
			V(z_n)^{2(1-\alpha)} \leq B_n \left[V(z_{n-1})-V(z_{n+1})\right].
		\end{equation*}
		Then we will apply the Gronwall lemma \ref{Gronwall_two_step} with $w_n \coloneqq\frac{1}{B_n}$, $u_n \coloneqq V(z_n)$ and $\gamma \coloneqq 2(1-\alpha)$. If $\alpha=\frac{1}{2}$, for all $n\geq n_1+1$, we get:
		\begin{equation}
			V(z_n) \leq e^{\sqrt{V(z_{n_1}) V(z_{n_1+1})}} \exp{\left(-\frac{1}{2}\displaystyle{\sum_{k=n_1}^{n-1}\frac{1}{B_k}}\right)}.
			\label{momentum_V_exp}
		\end{equation}
		If $0<\alpha\leq \frac{1}{2}$, for all $n\geq n_1+1$, we get:
		\begin{equation}
			V(z_n) \leq \left[\dfrac{2}{V(z_{n_1})^{2\alpha-1}+V(z_{n_1+1})^{2\alpha-1}+(1-2\alpha)\displaystyle{\sum_{k=n_1}^{n-1}\frac{1}{B_k}}}\right]^{\frac{1}{1-2\alpha}}.
			\label{momentum_V_subexp}
		\end{equation}
		If $\frac{1}{2}<\alpha<1$, if $\displaystyle{\sum_{k\geq 0}} \frac{1}{B_k}=+\infty$, then $\left(V(z_n)\right)_{n\in \mathbb{N}}$ is stationary, equal to $V(z^*)=0$. It is indeed the case: $B_n$ involves the term $\eta_{n-1}$ and $\eta_n$ that are lower and upper bounded by positive constants so $B_n$ is upper bounded by a strictly positive constant, see \eqref{expr_Bn}. Then $\inf \frac{1}{B_n}>0$ which is sufficient to conclude that the sum diverges. \\
		
		The Lyapunov inequality \eqref{Momentum_condition} combined with $\|v_n\| = \frac{\Delta_n}{\eta_{n-1}}$ leads directly to:
		\begin{equation}
			\Delta_n \leq \sqrt{\frac{\eta_{n-1}}{\lambda \bar{\beta}}}\sqrt{V(z_{n-1})-V(z_n)} \leq \sqrt{\frac{\eta_{n-1}}{\lambda \bar{\beta}}}\sqrt{V(z_{n-1})} \leq \frac{1}{\bar{\beta}}\sqrt{\frac{1}{\lambda}}\sqrt{V(z_{n-1})}.
			\label{momentum_delta_infV}
		\end{equation}
		since $\eta_{n-1} \leq \frac{1}{\bar{\beta}}$. Combining \eqref{theta_V} with the previous inequality, we get:
		\begin{equation*}
			\|\theta_n-\theta^*\| \leq \dfrac{\Delta_n}{2}+A^*V(z_n)^{\alpha} \leq \frac{1}{\bar{2\beta}}\sqrt{\frac{1}{\lambda}}\sqrt{V(z_{n-1})}+A^*V(z_n)^{\alpha}.
		\end{equation*}
		To obtain the convergences rates for $\alpha \leq \frac{1}{2}$, it is sufficient to inject inequalities \eqref{momentum_V_exp} and \eqref{momentum_V_subexp} in the previous one. If $\frac{1}{2}<\alpha<1$, for $n$ large enough, $V(z_n)=0$ as well as $\Delta_n=0$ by \eqref{momentum_delta_infV}. By the previous inequality, we also conclude that $\theta_n=\theta^*$ for $n$ large enough.
	\end{proofpart}
\end{proof}

\begin{proof}[Proof of corollary \ref{momentum_corollary}]
	The lemma \ref{Momentum_eta_inf} tells us that if one takes $\eta_n = \eta_o$ for all $n\geq 0$, then the dissipation inequality \eqref{Momentum_condition} is satisfied for all the iterations. Then we can just copy the proof above.
\end{proof}

\section{RMSProp proof}
\label{appendix_rms}

In this section, we prove theorem \ref{rms_th} and corollary \ref{rms_corollary}. We will first introduce some lemmas similar to the ones for Momentum. 

In the same manner as Momentum, we begin by proving that the Lyapunov condition \eqref{RMSProp_Lyapunov} can be achieved by a constant step size policy. 

For $\theta,s \in \Rb^N$ and a fix $\eta$, let us introduce the notations $s_{\eta}\coloneqq (1-\bar{\beta}\eta)s+\bar{\beta}\eta \nR(\theta)^2$ and $\theta_{\eta}=\theta - \eta \dfrac{\nR(\theta)}{\epsilon_a+\sqrt{s_{\eta}}}$.
\begin{lemma}
	Assume \ref{c1}, \ref{local_lip} and consider a compact set $K\subset \Rb^N$. Then for all $0<\eta\leq \eta_o$ \eqref{def_eta_star} (with $L$ the Lipshitz constant on $K$) and $\theta \in K$:
	\begin{equation*}
		\R(\theta_{\eta})-\R(\theta) \leq - \lambda \eta \left|\left|\dfrac{\nR(\theta)}{\sqrt{\epsilon_a+\sqrt{s_{\eta}}}}\right|\right|^2.
	\end{equation*}
	\label{RMSProp_eta}
\end{lemma}

\begin{proof}
	As $\nR$ is locally Lipshitz, it is globally Lipshitz on $K$ (denoting by $L$ the lipshitz constant on $K$). By applying the smoothness inequality as in the proof of lemma \ref{Momentum_eta_inf}, we get:
	\begin{multline*}
		\R(\theta_{\eta}) \leq \R(\theta) - \eta \nR(\theta)\cdot \dfrac{\nR(\theta)}{\epsilon_a+\sqrt{s_{\eta}}} + \frac{L\eta^2}{2}\left|\left|\dfrac{\nR(\theta)}{\epsilon_a+\sqrt{s_{\eta}}}\right|\right|^2 \\
		\leq \eta \left[\frac{L\eta}{2}\left|\left|\dfrac{\nR(\theta)}{\epsilon_a+\sqrt{s_{\eta}}}\right|\right|^2-\nR(\theta)\cdot \dfrac{\nR(\theta)}{\epsilon_a+\sqrt{s_{\eta}}}\right]. 
	\end{multline*}
	We can bound the previous norm with the dissipation term (right hand side of the Lyapunov inequality \eqref{RMSProp_Lyapunov}):
	\begin{multline*}
		\left|\left|\dfrac{\nR(\theta)}{\epsilon_a+\sqrt{s_{\eta}}}\right|\right|^2 = \left|\left|\dfrac{1}{\sqrt{\epsilon_a+\sqrt{s_{\eta}}}} \dfrac{\nR(\theta)}{\sqrt{\epsilon_a+\sqrt{s_{\eta}}}}\right|\right|^2 \\
		\leq \lambda_{max}\left(diag\left(\dfrac{1}{\sqrt{\epsilon_a+\sqrt{s_{\eta}}}}\right)\right)^2 \left|\left|\dfrac{\nR(\theta)}{\sqrt{\epsilon_a+\sqrt{s_{\eta}}}}\right|\right|^2 \\ 
		\leq \dfrac{1}{\left[\min_{1\leq i \leq N}\sqrt{\epsilon_a+\sqrt{s_{\eta}}}\right]^2}\left|\left|\dfrac{\nR(\theta)}{\sqrt{\epsilon_a+\sqrt{s_{\eta}}}}\right|\right|^2 \leq \frac{1}{\epsilon_a} \left|\left|\dfrac{\nR(\theta)}{\sqrt{\epsilon_a+\sqrt{s_{\eta}}}}\right|\right|^2.
	\end{multline*}
	In the lines above, $\lambda_{max}(diag(x))$ denotes the biggest eigenvalue of the diagonal matrix with a diagonal given by the vector $x\in \Rb^N$.
	Combining the two points below, we get:
	\begin{equation*}
		\R(\theta_{\eta})-\R(\theta) \leq \eta \left[\frac{L\eta}{2}\frac{1}{\epsilon_a}-1 \right] \left|\left|\dfrac{\nR(\theta)}{\sqrt{\epsilon_a+\sqrt{s_{\eta}}}}\right|\right|^2.
	\end{equation*}
	To satisfy the Lyapunov condition \eqref{RMSProp_Lyapunov} it is sufficient to require that:
	\begin{equation*}
		\frac{L\eta}{2}\frac{1}{\epsilon_a}-1  \leq -\lambda \Leftrightarrow \eta \leq \frac{2\epsilon_a}{L}(1-\lambda).
	\end{equation*}
	The inequality $\eta_o \leq \frac{2\epsilon_a}{L}(1-\lambda)$ concludes the proof.
\end{proof}

Contrary to Momentum, an extra step to bound the sequence $(s_n)_{n\in \mathbb{N}}$ is needed. The next two lemmas are here to bound $(s_n)_{n\in \mathbb{N}}$ for gradient that lies in the unit ball. Let us denote by $s_n^i$ the $i$-th component of the vector $s_n$ for $1\leq i \leq N$.

\begin{lemma}
	For all $n\in \mathbb{N}$:
	\begin{equation*}
		\|\theta_{n+1}-\theta_n\| \leq \sqrt{\frac{\eta_n N}{\bar{\beta}}}.
	\end{equation*}
	\label{lemma_rms_dtheta}
\end{lemma}
\begin{proof}
	For all $1 \leq i \leq N$:
	\begin{equation*}
		\sqrt{s_{n+1}^i} \geq \sqrt{\beta_n s_n^i + (1-\beta_n)\partial_i \R(\theta_n)^2} \geq \sqrt{1-\beta_n}|\partial_i \R(\theta_n)|.
	\end{equation*}
	Then we can write:
	\begin{equation*}
		|\theta_{n+1}^i-\theta_n^i| = \dfrac{\eta_n |\partial_i \R(\theta_n)|}{\epsilon_a+\sqrt{s_{n+1}^i}} \leq \dfrac{\eta_n |\partial_i \R(\theta_n)|}{\epsilon_a+\sqrt{1-\beta_n}|\partial_i \R(\theta_n)|} \leq \frac{\eta_n}{\sqrt{1-\beta_n}} \leq \sqrt{\frac{\eta_n}{\bar{\beta}}}.
	\end{equation*}
	By summing up these inequalities over $i$, we get:
	\begin{equation*}
		\|\theta_{n+1}-\theta_n\|^2 \leq \displaystyle{\sum_{i=1}^N}\frac{\eta_n}{\bar{\beta}} \leq \frac{\eta_n}{\bar{\beta}}N. 
	\end{equation*}
	By taking the square root in the previous inequality, we derive the expected bound.
\end{proof}

\begin{lemma}
	Assume \ref{c1} and \ref{local_lip}. If for some $n\geq 0$, $\|\nR(\theta_n)\|_{\infty} \leq 1$, then for all $1 \leq i \leq N$, $\left(\epsilon_a+\sqrt{s_{n+1}^i}\right)^2 \leq S_{n+1} \leq S$. Here $S_n$ and $S$ are given at the beginning of section \ref{analysis_rms}.
	\label{RMSProp_s}
\end{lemma}
\begin{proof}
	By a direct induction, we obtain:
	\begin{equation*}
		s_{n+1}^i = \displaystyle{\sum_{j=0}^n} B_{n-j} \partial_i\R(\theta_{n-j})^2,
	\end{equation*}
	where we recall that for $0\leq j \leq n$:
	\begin{equation*}
		B_{n-j} = (1-\beta_{n-j})\displaystyle{\prod_{l=n-j+1}^n}\beta_l.
	\end{equation*}
	For $j>n$, we consider $B_{n-j}=0$. We have to express $\partial_i\R(\theta_{n-j})$ in function of $\partial_i\R(\theta_n)$. Let us recall that $K=\overline{\left\{\theta_n, n\geq 0\right\}}$ is a compact set and consider the lipshitz constant $L$ of $\nR$ on $K$. By triangle inequality and the Lipshitz property, we have:
	\begin{multline*}
		|\partial_i \R(\theta_{n-j})| \leq \displaystyle{\sum_{k=0}^{j-1}}|\partial_i \R(\theta_{n-k-1}) - \partial_i \R(\theta_{n-k})| + |\partial_i \R(\theta_n)| \\ \leq L\displaystyle{\sum_{k=0}^{j-1}}\|\theta_{n-k-1}-\theta_{n-k}\| + |\partial_i \R(\theta_n)| \leq L\displaystyle{\sum_{k=0}^{j-1}} \sqrt{\frac{\eta_{n-k-1}}{\bar{\beta}}N} + |\partial_i \R(\theta_n)|.
	\end{multline*}
	The last inequality is a direct consequence of lemma \ref{lemma_rms_dtheta}. Combining this inequality with the explicit formula of $s_{n+1}^i$ leads to:
	\begin{equation*}
		s_{n+1}^i \leq \displaystyle{\sum_{j=0}^n}B_{n-j}\left[|\partial_i \R(\theta_n)| + L\displaystyle{\sum_{k=0}^{j-1}} \sqrt{\frac{\eta_{n-k-1}}{\bar{\beta}}N} \right]^2.
	\end{equation*}
	By developing the terms in the brackets, we get:
	\begin{multline}
		s_{n+1}^i \leq \partial_i \R(\theta_n)^2 \displaystyle{\sum_{j=0}^\infty}B_{n-j} + 2L\sqrt{\frac{N}{\bar{\beta}}} |\partial_i \R(\theta_n)| \displaystyle{\sum_{j=0}^\infty}B_{n-j}\left(\displaystyle{\sum_{k=0}^{j-1}}\sqrt{\eta_{n-k-1}}\right) \\ 
		+ \frac{L^2N}{\bar{\beta}} \displaystyle{\sum_{j=0}^\infty}B_{n-j}\left(\displaystyle{\sum_{k=0}^{j-1}} \sqrt{\eta_{n-k-1}}\right)^2,
		\label{bound_sn}
	\end{multline}
	provided that the sums below converge (this will be shown later).
	As $\|\nR(\theta_n)\|_{\infty} \leq 1$, $s_{n+1}^i \leq \tilde{S}_{n+1}$. 
	This shows the first part of the lemma, that is to say, $\left(\epsilon+\sqrt{s_{n+1}^i}\right)^2 \leq S_{n+1}$.\\
	
	Reasoning in the same manner as part 1 of the proof of theorem \ref{momentum_th}, we get $n_0\geq 0$ such that $\displaystyle{\inf_{n\geq n_0}} \eta_n = \eta^* >0$, by using lemma \ref{RMSProp_eta} instead of \ref{Momentum_eta_inf}. Without loss of generality, we can assume $n_0=0$ to avoid introduce new notations. As for all $n\geq 0$, $\eta^* \leq \eta_n \leq \frac{1}{\bar{\beta}}$ the following holds for $0\leq j \leq n$:
	\begin{equation}
		B_{n-j} \leq \left(1-\bar{\beta}\eta^*\right)^j.
		\label{Bn_ineq}
	\end{equation}
	Let us note that $0 \leq 1-\bar{\beta}\eta^* < 1$. Given that the sum $\displaystyle{\sum_{n\geq 0}} x^n$ and its derivatives converge for $|x|<1$, the sums involved above converge. Let us recall the following classical sums for $|x|<1$:
	\begin{equation*}
		\begin{array}{ll}
			\displaystyle{\sum_{j\geq 0}{x^j}} = \dfrac{1}{1-x}, \\
			\displaystyle{\sum_{j\geq 0}{j x^j}} = \dfrac{x}{(1-x)^2}, \\
			\displaystyle{\sum_{j\geq 0}{j^2 x^j}} = \dfrac{x(1+x)}{(1-x)^3}
		\end{array}
	\end{equation*}
	Using these identities with $x=1-\bar{\beta}\eta^*$ and \eqref{Bn_ineq}, we get:
	\begin{equation*}
		\begin{array}{ll}
			\displaystyle{\sum_{j=0}^\infty}B_{n-j} \leq \displaystyle{\sum_{j=0}^\infty}(1-\bar{\beta}\eta^*)^j = \frac{1}{\bar{\beta}\eta^*}, \\
			2L\sqrt{\frac{N}{\bar{\beta}}} \displaystyle{\sum_{j=0}^\infty}B_{n-j}\left(\displaystyle{\sum_{k=0}^{j-1}}\sqrt{\eta_{n-k-1}}\right) \leq \frac{2L\sqrt{N}}{\bar{\beta}} \displaystyle{\sum_{j=0}^\infty}j(1-\bar{\beta}\eta^*)^j = \dfrac{2L\sqrt{N}(1-\bar{\beta}\eta^*)}{\bar{\beta}^3(\eta^*)^2}, \\
			\frac{L^2N}{\bar{\beta}} \displaystyle{\sum_{j=0}^\infty}B_{n-j}\left(\displaystyle{\sum_{k=0}^{j-1}} \sqrt{\eta_{n-k-1}}\right)^2 \leq \frac{L^2N}{\bar{\beta}^2} \displaystyle{\sum_{j=0}^\infty}j^2(1-\bar{\beta}\eta^*)^j \leq \dfrac{L^2N (1-\bar{\beta}\eta^*)(2-\bar{\beta}\eta^*)}{\bar{\beta}^5 (\eta^*)^3}. 
		\end{array}
	\end{equation*}
	To finish the proof, it is sufficient to inject the three inequalities above in the expression of $\tilde{S}_{n+1}$.
\end{proof}

\begin{remark}
	Adapting the proof above, we can show the following extension.
	If for some $M>0$ and $n\geq 0$, $\|\nR(\theta_n)\|_{\infty} \leq M$, then it exists $S_M>0$ such that: for all $1 \leq i \leq N$, $\left(\epsilon_a+\sqrt{s_{n+1}^i}\right)^2 \leq S_M$. To do this it is sufficient to replace $|\partial_i \R(\theta_n)|$ by $M$ in the inequality \eqref{bound_sn}.
	\label{bound_s_general}
\end{remark}

As for Momentum, we need a discrete Gronwall lemma but for only one step sequence.

\begin{lemma}[Gronwall inequality]
	Let $(w_n)_{n\in \mathbb{N}}$ be a positive sequence and $(u_n)_{n\in \mathbb{N}}$ a sequence that satisfies:
	\begin{equation}
		u_{n+1}-u_n \leq -w_n u_n.
	\end{equation}
	Then, the following estimate holds:
	\begin{equation}
		u_n \leq u_0 \exp{\left( -\displaystyle{\sum_{k=0}^{n-1}} w_k \right)}.
	\end{equation}
	\label{gronwall_exp}
\end{lemma}
\begin{proof}
	It is a classical lemma in numerical and stability analysis, see chapter 1 (section 1.9) of \cite{gronwall_stability}.
\end{proof}

\begin{lemma}
	Let $(w_n)_{n\in \mathbb{N}}$ be a positive sequence, $0<\gamma\neq 1$ and $(u_n)_{n\in \mathbb{N}}$ a lower bounded sequence satisfying:
	\begin{equation}
		u_{n+1}-u_n \leq -w_n u_n^{\gamma}.
	\end{equation}
	If $\gamma>1$, we have:
	\begin{equation}
		u_n \leq \dfrac{1}{\left[ u_0^{1-\gamma} +(\gamma-1) \displaystyle{\sum_{k=0}^{n-1}} w_k \right]^{\frac{1}{\gamma-1}}}.
	\end{equation}
	If $0<\gamma<1$ and $\displaystyle{\sum_{k\geq 0} w_k}$ diverges, $(u_n)_{n\in \mathbb{N}}$ is stationary.
	\label{gronwall_power}
\end{lemma}
\begin{proof}
	Exactly the same proof as the two step Gronwall lemma \ref{Gronwall_two_step} since $(u_n)_{n\in \mathbb{N}}$ is decreasing by the inequality $u_{n+1}-u_n \leq -w_n u_n^{\alpha} \leq 0$.
\end{proof}

\begin{proof}[Proof of theorem \ref{rms_th}]
	This proof is organized in three parts. In the first one, we show that $\left(\nR(\theta_n)\right)_{n\in \mathbb{N}}$ converges to $0$. This makes it possible to extract a subsequence $\left(\theta_{\phi(n)}\right)_{n\in \mathbb{N}}$ that converges to a critical point $\theta^*$. Then in the second part, we establish an essential inequality similar to the lemma \ref{lemma_delta_momentum}. concerning the distance between two successive iterates, using lemma \ref{RMSProp_s}. This enables to prove that the sequence $(\theta_n)_{n\in \mathbb{N}}$ can not leave an arbitrarily small neighborhood around $\theta^*$ as soon as the subsequence $\left(\theta_{\phi(n)}\right)_{n\in \mathbb{N}}$ enters this neighborhood. This will lead to the convergence of $(\theta_n)_{n\in \mathbb{N}}$ and conclude the second part. Combining the Lyapunov inequality \eqref{RMSProp_Lyapunov} with the lemma \ref{RMSProp_s}, we establish a Gronwall type inequality on $\R(\theta_n)$. Then lemma \ref{gronwall_power} gives convergence rates on $\R(\theta_n)$. Combining this with the central inequality proved in part 2, the convergence rates of theorem \ref{rms_th} are derived.
	
	\begin{proofpart}
		We already know that $(\theta_n)_{n\in \mathbb{N}}$ is bounded thanks to the assumption \ref{coercive} (see the discussion at the beginning of section \ref{analysis_rms}). In this part, we will prove that \textbf{this sequence admits a subsequence that converges to a critical point $\theta^*$ of $\R$}. To do this, we need to prove that $\left(\nR(\theta_n)\right)_{n\in \mathbb{N}}$ converges to $0$. \\
		As the sequence $\left(\R(\theta_n)\right)_{n\in \mathbb{N}}$ is decreasing due to the dissipation inequality \eqref{RMSProp_Lyapunov} and bounded by below, it converges. By passing to the limit in inequality \eqref{RMSProp_Lyapunov}, we get:
		\begin{equation*}
			\displaystyle{\lim_{n\to \infty}} \eta_n \left|\left|\dfrac{\nR(\theta_n)}{\sqrt{\epsilon_a+\sqrt{s_{n+1}}}}\right|\right|^2 =0.
		\end{equation*}
		As $\inf \eta_n >0$ (see the proof of lemma \ref{RMSProp_s}), we can write:
		\begin{equation*}
			\displaystyle{\lim_{n\to \infty}} \left|\left|\dfrac{\nR(\theta_n)}{\sqrt{\epsilon_a+\sqrt{s_{n+1}}}}\right|\right|^2 =0.
		\end{equation*}
		The sequence $(\nR(\theta_n))_{n\in \mathbb{N}}$ is bounded by some $M>0$ (since $(\theta_n)_{n\in \mathbb{N}}$ is bounded) so by remark \ref{bound_s_general} it exists $S_M>0$ such that for all $n\geq 0$ and $1 \leq i \leq N$: $\epsilon_a+\sqrt{s_{n+1}} \leq \sqrt{S_M}$. The following inequality:
		\begin{equation*}
			\left|\left|\dfrac{\nR(\theta_n)}{\sqrt{\epsilon_a+\sqrt{s_{n+1}}}}\right|\right|^2 \geq \dfrac{\|\nR(\theta_n)\|^2}{\sqrt{S_M}},
		\end{equation*}
		gives $\nR(\theta_n) \to 0$. 
		As $(\theta_n)_{n\in \mathbb{N}}$ is bounded, it exists a subsequence $(\theta_{\phi(n)})_{n\in \mathbb{N}}$ that converges to a point $\theta^* \in \Rb^N$. By continuity of $\nR$ (consequence of assumption \ref{local_lip}) and the fact that $\nR(\theta_n) \to 0$, $\nR(\theta^*)=0$ is satisfied. 
	\end{proofpart}
	
	\begin{proofpart}
		In this part, we show that the \textbf{sequence $(\theta_n)_{n\in \mathbb{N}}$ converges. To do this, we need to focus on the sum $\displaystyle{\sum_{n\geq n_0}}\Delta_{n+1}$} where $\Delta_{n+1}\coloneqq \|\theta_{n+1}-\theta_n\|$. Let us denote by $\mathcal{U}$ the neighborhood of $\theta^*$ where the Lojasiewicz inequality is valid. Without loss of generality, we will assume $\R(\theta^*)=0$. \\\\
		We get the following inequality in the same way as in the proof of Momentum:
		\begin{equation}
			\R(\theta_n)^{\alpha}-\R(\theta_{n+1})^{\alpha} \geq \dfrac{\alpha c}{\|\nR(\theta_n)\|} \left[\R(\theta_n)-\R(\theta_{n+1})\right]
			\label{R_alpha}
		\end{equation}
		for $\theta_n \in \mathcal{U}$. Next, we have to give a lower bound for $\left[\R(\theta_n)-\R(\theta_{n+1})\right]$ and an upper bound for the gradient. By the update rule \eqref{RMSProp_eq} and the Lyapunov condition \eqref{RMSProp_Lyapunov} respectively, we can write:
		\begin{equation*}
			\begin{array}{ll}
				\nR(\theta_n) = \dfrac{\epsilon_a+\sqrt{s_{n+1}}}{\eta_n}(\theta_n-\theta_{n+1})  \\
				\R(\theta_n)-\R(\theta_{n+1}) \geq \lambda \eta_n \displaystyle{\sum_{i=1}^N} \dfrac{\partial_i \R(\theta_n)^2}{\epsilon_a+\sqrt{s_{n+1}^i}} \geq \lambda \displaystyle{\sum_{i=1}^N} \dfrac{\epsilon_a+\sqrt{s_{n+1}^i}}{\eta_n}(\theta_{n+1}^i-\theta_n^i)^2 \geq \dfrac{\lambda \epsilon_a}{\eta_n}\Delta_{n+1}^2.
			\end{array}
		\end{equation*}
		From lemma \ref{RMSProp_s}, we can deduce that for $\theta_n \in \mathcal{V}\coloneqq \mathcal{U} \cap \{\theta \in \Rb^N, \|\nR(\theta)\|_{\infty}<1\}$:
		\begin{equation*}
			\|\nR(\theta_n)\| \leq \frac{\sqrt{S_{n+1}}}{\eta_n}\Delta_{n+1}.
		\end{equation*}
		Injecting the two last inequalities above in \eqref{R_alpha}, we get for $\theta_n \in \mathcal{V}$:
		\begin{equation*}
			\Delta_{n+1} \leq \dfrac{\sqrt{S_{n+1}}}{\lambda \alpha c \epsilon_a} \left[\R(\theta_n)^{\alpha}-\R(\theta_{n+1})^{\alpha}\right] \leq \dfrac{\sqrt{S}}{\lambda \alpha c \epsilon_a} \left[\R(\theta_n)^{\alpha}-\R(\theta_{n+1})^{\alpha}\right].
		\end{equation*}
		By summing the previous inequality from $k=p$ to $k=q-1$, we get for $\theta_p, \dots, \theta_{q-1} \in \mathcal{V}$:
		\begin{equation}
			\displaystyle{\sum_{k=p}^{q-1}} \Delta_k \leq \dfrac{\sqrt{S}}{\lambda \alpha c \epsilon_a} \left[\R(\theta_p)^{\alpha}-\R(\theta_{q})^{\alpha}\right].
			\label{rms_sum_delta}
		\end{equation}
		
		Let $r>0$ be such that $B_r(\theta^*) \subset \mathcal{V}$. Given that $\theta^*$ is a accumulation point of $(\theta_n)_{n\in \mathbb{N}}$ and $(\R(\theta_n))_{n\in \mathbb{N}}$ converges, it exists $n_1 \geq 0$ such that:
		\begin{equation}
			\begin{array}{lll}
				\|\theta_{n_1}-\theta^*\|<\frac{r}{2}, \\\\
				\forall q \geq n_1: \dfrac{\sqrt{S}}{\lambda \alpha c \epsilon_a} \left[\R(\theta_{n_1})^{\alpha}-\R(\theta_{q})^{\alpha}\right] < \frac{r}{2}, \\\\
			\end{array}
			\label{rms_r_cv}
		\end{equation}
		
		Show that $\theta_n \in B_r(\theta^*)$ for all $n\geq n_1$. By contradiction assume that it is not the case: $\exists n\geq n_1, \theta_n \notin B_r(\theta^*)$. The set $\left\{ n\geq n_1, \theta_n \notin B_r(\theta^*)\right\}$ is a non empty bounded by below part of $\mathbb{N}$. So we can consider the minimum of this set that we denote by $q$. As a result, $\forall n_1 \leq n < q$, $\theta_n \in V$ so we can apply \eqref{rms_sum_delta} with \eqref{rms_r_cv} to get:
		\begin{equation*}
			\displaystyle{\sum_{n=n_1}^{q-1}}\Delta_{n} \leq \dfrac{\sqrt{S}}{\lambda \alpha c \epsilon} \left[\R(\theta_{n_1})^{\alpha}-\R(\theta_{q})^{\alpha}\right] \leq \frac{r}{2}. 
		\end{equation*}
		This implies:
		\begin{equation*}
			\|\theta_q-\theta^*\| \leq \displaystyle{\sum_{n=n_1}^{q-1}}\Delta_{n+1} + \|\theta_{n_1}-\theta^*\| < r.
		\end{equation*}
		This is a contradiction because $\|\theta_q-\theta^*\| \geq r$. As $r$ is arbitrary small this shows the convergence of $(\theta_n)_{n\in \mathbb{N}}$.
	\end{proofpart}
	
	\begin{proofpart}
		To get an \textbf{estimation of the convergence rate}, we are looking for a relation between the gradient and the difference $\left[\R(\theta_{n+1})-\R(\theta_n)\right]$. \\\\
		Using the dissipation inequality \eqref{RMSProp_Lyapunov} with the lemma \ref{RMSProp_s}, we have for $n\geq n_1$ (as $\theta_n \in \mathcal{V}$):
		\begin{equation*}
			\R(\theta_n)-\R(\theta_{n+1}) \geq \lambda \eta_n \displaystyle{\sum_{i=1}^N} \frac{\partial_i \R(\theta_n)^2}{\epsilon_a+\sqrt{s_{n+1}^i}}  \geq \dfrac{\lambda \eta_n}{\sqrt{S_{n+1}}} \|\nR(\theta_n)\|^2.
		\end{equation*}
		For $n\geq n_1$, we can apply Lojasiewicz inequality (as $\mathcal{V} \subset \mathcal{U}$) to obtain:
		\begin{equation*}
			\R(\theta_n)-\R(\theta_{n+1}) \geq \dfrac{\lambda c^2 \eta_n}{\sqrt{S_{n+1}}} \R(\theta_n)^{2(1-\alpha)}.
		\end{equation*}
		Now we apply lemmas \ref{gronwall_exp} and \ref{gronwall_power} on the previous inequality with $u_n\coloneqq\R(\theta_n)$, $w_n\coloneqq\dfrac{\lambda c^2 \eta_n}{\sqrt{S_{n+1}}}$ and $\gamma\coloneqq 2(1-\alpha)$. This gives:
		\begin{itemize}
			\item if $\alpha=\frac{1}{2}$ and $n\geq n_1$, we have:
			\begin{equation*}
				\R(\theta_n) \leq \R(\theta_{n_1}) \exp{\left( -\displaystyle{\sum_{k=n_1}^{n-1}} \dfrac{\lambda c^2 \eta_k}{\sqrt{S_{k+1}}} \right)}.
			\end{equation*}
			\item If $0<\alpha<\frac{1}{2}$ and $n\geq n_1$, we get:
			\begin{equation*}
				\R(\theta_n) \leq \dfrac{1}{\left[ \R(\theta_0)^{1-\alpha} +(\alpha-1) \displaystyle{\sum_{k=0}^{n-1}} \dfrac{\lambda c^2 \eta_k}{\sqrt{S_{k+1}}} \right]^{\frac{1}{\alpha-1}}}.
			\end{equation*}
			\item If $\frac{1}{2}<\alpha<1$, then $\left(\R(\theta_n)\right)_{n\in \mathbb{N}}$ is stationary equal to $\R(\theta^*)=0$ since the sum $\displaystyle{\sum_{k\geq 0}} \dfrac{\eta_k}{\sqrt{S_{k+1}}} = +\infty$. Indeed, for $n \geq n_1$, $\dfrac{\eta_n}{\sqrt{S_{n+1}}} \geq \dfrac{\eta^*}{\sqrt{S}}>0$.
		\end{itemize}
		We finish the proof by combining the three cases above with the following inequality for $n\geq n_1$ (use \eqref{rms_sum_delta} for the second inequality):
		\begin{equation*}
			\|\theta_n-\theta^*\| \leq \displaystyle{\sum_{k=n}^{+\infty}} \Delta_{k+1} \leq \dfrac{\sqrt{S}}{\lambda \alpha c \epsilon_a} \R(\theta_n)^{\alpha}.
		\end{equation*}
	\end{proofpart}
	
\end{proof}

\section{Complexity proof}
\label{appendix_complexity}

\begin{proof}[Proof of theorem \ref{complexity_theorem}]
	Using the notations of lemma \ref{Momentum_eta_inf} and \ref{RMSProp_eta}, we define for Momentum and RMSProp respectively:
	\begin{equation*}
		f(\theta,\eta) = V(\theta_{\eta},v_{\eta})-V(\theta,v)+\lambda \bar{\beta}\|v_{\eta}\|^2. 
	\end{equation*}
	\begin{equation*}
		f(\theta,\eta) = \R(\theta_{\eta})-\R(\theta)+\lambda \eta \left|\left|\dfrac{\nR(\theta)}{\sqrt{\epsilon_a+\sqrt{s_{\eta}}}}\right|\right|^2. 
	\end{equation*}
	The learning rates chosen by the two algorithms satisfy:
	\begin{equation}
		\begin{array}{cc}
			\eta_0 = \dfrac{1}{\bar{\beta}f_1^{p_0}}, \\
			\eta_n = \dfrac{\min{\left(f_2 \eta_{n-1}, 1/\bar{\beta}\right)}}{f_1^{p_n}} \text{ for } n\geq 1, 
		\end{array}
		\label{eta_n_formel}
	\end{equation}
	where:
	\begin{equation*}
		\begin{array}{cc}
			p_0 = \min \left\{p\geq 0, f\left(\theta_n, \frac{1}{\bar{\beta}f_1^p} \right) \leq 0\right\}, \\
			p_n = \min \left\{p\geq 0, f\left(\theta_n, \dfrac{\min{\left(f_2 \eta_{n-1}, 1/\bar{\beta}\right)}}{f_1^p}\right) \leq 0\right\} \text{ for } n \geq 1.
		\end{array}
	\end{equation*}
	By applying the log function to \eqref{eta_n_formel}, it gives immediately:
	\begin{equation*}
		\begin{array}{cc}
			p_0 = \dfrac{\log{\left(\frac{1}{\bar{\beta}\eta_0}\right)}}{\log{(f_1)}}, \\
			p_n \leq  \dfrac{\log{(f_2)}}{\log{(f_1)}} + \dfrac{1}{\log{(f_1)}} \left[\log{(\eta_{n-1})}-\log{(\eta_n)}\right].
		\end{array}
	\end{equation*}
	We have one evaluation to begin the algorithm $\left(\R(\theta_0) \text{ or } V(\theta_0)\right)$ and at the $k$-th ($k\geq 0$) iteration we evaluate $f(\theta_k, f_2\eta_{k-1})$ until $f\left(\theta_k, \frac{f_2\eta_{k-1}}{f_1^{p_k}}\right)$ which corresponds to $p_k+1$ function evaluations. Then a telescopic summation leads to:
	\begin{equation*}
		C_n = 1+ (1+p_0) + \displaystyle{\sum_{k=1}^n} (p_k+1) \leq 2 + \dfrac{\log{\left(\frac{1}{\bar{\beta}\eta_0}\right)}}{\log{(f_1)}} + \left[1+\dfrac{\log{(f_2)}}{\log{(f_1)}}\right]n + \dfrac{1}{\log{(f_1)}} \left[\log{(\eta_0)}-\log{(\eta_n)}\right].
	\end{equation*}
	We know from the proof of theorems \ref{momentum_th} and \ref{rms_th} that for $n$ big enough, $\eta_n \geq \eta^*$.
	\begin{multline*}
		C_n \leq 2 + \dfrac{\log{\left(\frac{1}{\bar{\beta}\eta_0}\right)}}{\log{(f_1)}} + \left[1+\dfrac{\log{(f_2)}}{\log{(f_1)}}\right]n + \dfrac{1}{\log{(f_1)}} \left[\log{(\eta_0)}-\log{(\eta^*)}\right] \\
		= \left[1+\dfrac{\log{(f_2)}}{\log{(f_1)}}\right]n + \dfrac{\log{\left(\frac{f_1^2}{\bar{\beta}\eta^*}\right)}}{\log{(f_1)}}.
	\end{multline*}
	By dividing by $n$, we get the result.
\end{proof}

\begin{remark}
	We really think that the fact that $\eta^*$ (and so the stiffness and the hyperparameters) only appears in the deviation term is due to the memory effect, that is to say the fact that the backtracking uses the previous time step value to begin a new linesearch. Then it is a relevant question to know if one can establish the same result than theorem \ref{complexity_theorem} for the classical backtracking \cite{armijo}.  
\end{remark}
				
\end{document}